\documentclass[preprint,12pt]{elsarticle}

\usepackage{amssymb}
\usepackage{amsmath}

\usepackage[top=1.1in, bottom=1.2in, left=0.8in, right=0.8in]{geometry}
\usepackage{amsmath,amsfonts,amssymb,amsthm,bm}
\usepackage{mathrsfs,dsfont}
\usepackage{mathtools}
\usepackage{graphicx}
\usepackage{colortbl,dcolumn}
\usepackage{psfrag}
\usepackage{booktabs}
\usepackage{subfigure}
\usepackage{hyperref}
\usepackage{url}
\usepackage{comment}
\allowdisplaybreaks[4]
\numberwithin{equation}{section}

\newcommand{\R}{\mathbb{R}}
\newcommand{\N}{\mathbb{N}}
\newcommand{\E}{\mathbb{E}}

\renewcommand{\Re}{{\rm{Re}}}
\renewcommand{\Im}{{\rm{Im}}}
\newcommand{\diff}{{\,\rm{d}}}
\newcommand{\dif}{{\rm{d}}}

\renewcommand{\i}{\mathbf{i}}

\newcommand{\Ph}{\Phi}
\newcommand{\scrJ}{\mathcal{J}}
\newcommand{\scrK}{\mathcal{K}}
\newcommand{\lam}{\lambda}
\newcommand{\lamI}{\lam_{1}}
\newcommand{\lamII}{\lam_{2}}
\newcommand{\norm}[1]{\|#1\|}
\newcommand{\abs}[1]{|#1|}
\newtheorem{theorem}{Theorem}[section]

\newtheorem{lemma}[theorem]{Lemma}
\newtheorem{proposition}[theorem]{Proposition}

\newtheorem{remark}[theorem]{Remark}

\journal{Nuclear Physics B}

\begin{document}

\begin{frontmatter}



\title{Novel physical property preserved methods for stochastic Schr\"{o}dinger--KdV equation}


\author[labelCZH]{Ziheng Chen} 
\affiliation[labelCZH]{
            organization={School of Mathematics and Statistics},
            addressline={Yunnan University}, 
            city={Kunming},
            postcode={650500}, 
            state={Yunnan},
            country={China}}

\author[labelHJL1,labelHJL2]{Jialin Hong} 
\affiliation[labelHJL1]{
            organization={LSEC, ICMSEC, Academy of Mathematics and Systems Science},
            addressline={Chinese Academy of Sciences}, 
            city={Beijing}, 
            postcode={100190}, 
            state={Beijing},
            country={China}}
\affiliation[labelHJL2]{
            organization={School of Mathematical Sciences},
            addressline={University of Chinese Academy of Sciences}, 
            city={Beijing}, 
            postcode={100049}, 
            state={Beijing},
            country={China}}

\author[labelSLY]{Liying Sun} 
\affiliation[labelSLY]{
            organization={Academy for Multidisciplinary Studies},
            addressline={Capital Normal University}, 
            city={Beijing},
            postcode={100048}, 
            state={Beijing},
            country={China}}

\begin{abstract}

In this work, we study the stochastic Schr\"odinger--KdV equation driven by additive noise from both analytical and numerical viewpoints. We first establish the evolution laws for the averaged plasmon number, momentum, and energy, together with the conservation of the averaged particle number. Motivated by these intrinsic structures, we develop two temporal discretizations. One is constructed based on the splitting strategy and Crank--Nicolson scheme, and is shown to preserve the discrete evolution laws of the averaged plasmon number and momentum, as well as the discrete conservation law of the averaged particle number. 
The other is proposed within the constant scalar auxiliary variable framework, in which the nonlinear energy functional is reformulated so that a modified averaged energy law can be preserved at the discrete level.
Combining these temporal discretizations with a local discontinuous Galerkin approximation in space yields structure-preserving full discretizations inheriting the corresponding discrete physical laws. Numerical experiments are presented to validate the theoretical results and to demonstrate the accuracy, robustness, and effectiveness of the proposed methods.

\end{abstract}

%

\begin{keyword}
      Stochastic Schr\"{o}dinger--KdV equation
      \sep
      Averaged momentum evolution law 
      \sep
      Averaged energy evolution law 
      \sep
      Local discontinuous Galerkin method
      \sep
      Constant scalar auxiliary variable approach



\end{keyword}

\end{frontmatter}


\section{Introduction}\label{sec:introduction}

The coupled Schr\"{o}dinger--Korteweg--de Vries (KdV) system is a prototypical model in mathematical physics for describing the resonant interaction between short and long waves, where the short-wave component is governed by a nonlinear Schr\"{o}dinger equation and the long-wave component by a KdV equation. Such a model arises in a variety of physical settings, including fluid dynamics \cite{Benney1977}, nonlinear optics \cite{kivshar2003optical}, and plasma physics \cite{nishikawa1974coupled}. Consider the coupled Schr\"{o}dinger--KdV system
\begin{equation}
\label{eq:detsubsystem}
\left\{\begin{aligned}
      &\i \partial_t U + \partial_x^2 U 
      = \gamma_1 U V + \beta |U|^2 U, 
      \\
      &\partial_t V + \partial_x^3 V 
      = \gamma_2 \partial_x |U|^2 - V \partial_x V
\end{aligned}\right.    
\end{equation}
supplemented with the initial condition $(U(0),V(0)) = (U^{0},V^{0})$. Here, $U = U(t,x)$ denotes the complex-valued Schr\"{o}dinger component, while 
$V = V(t,x)$  stands for the real-valued KdV component.
A notable feature of system \eqref{eq:detsubsystem} is its rich Hamiltonian structure, which gives rise to several conserved quantities; see, e.g., \cite{he2024physical, wang2018high}. In particular, the following invariants 
\begin{itemize}
      \item the number of plasmon: 
      $\mathcal{D}(t) := \int_{\mathcal O} |U|^{2} \diff{x}
      = \mathcal{D}(0), \quad t \geq 0;
      $ 
      \item the number of particle: 
      $\mathcal{P}(t) := \int_{\mathcal O} V \diff{x} 
      = \mathcal{P}(0), \quad t \geq 0;$
      \item the momentum
      $\mathcal{M}(t) :=\int_{\mathcal O} \Big({\rm{Im}}\big(U
      \overline{\partial_{x}U}\big)
      + \frac{\gamma_{1}}{2\gamma_{2}}
      V^{2} \Big) \diff{x} = \mathcal{M}(0),
      \quad t \geq 0;$  
      \item and the energy: 
      $\mathcal{E}(t) :=
      \int_{\mathcal O}\Big(|\partial_{x}U|^{2}
      + \frac{\beta}{2}|U|^{4} + \gamma_{1}|U|^{2}V
      - \frac{\gamma_{1}}{6\gamma_{2}}V^{3} 
      + \frac{\gamma_{1}}{2\gamma_{2}}|
      \partial_{x}V|^{2}\Big)\diff{x}
      = \mathcal{E}(t), \quad t \geq 0,$  
\end{itemize}
are preserved. In practical applications, however, such deterministic dynamics are often influenced by random perturbations originating from thermal fluctuations, external forcing, or unresolved microscopic effects. To account for these uncertainties, one is naturally led to the stochastic Schr\"{o}dinger--KdV (SSKdV) equation \eqref{eq:SKeq}, which takes the form of a coupled system of stochastic partial differential equations driven by additive noise \cite{chen2025global, guo2008attractor}.

Although substantial progress has been made in the numerical treatment of the stochastic Schr\"{o}dinger equation and the stochastic KdV equation, especially in the development of structure-preserving discretizations (see, e.g., \cite{chen2016symplectic, chen2017meansquare, cui2018analysis, cui2017strong, cui2017stochastic,  cui2026effects, dambrosion2026stochastic, hong2024novel, kong2026structure, li2020ultraweak, liu2013mass, liu2024structure}), the corresponding theory and methodology for coupled stochastic dispersive systems are still relatively underdeveloped. In particular, for the stochastic Schr\"{o}dinger--KdV equation, the  evolution laws of averaged key physical quantities and its underlying geometric structures remain largely unexplored. Meanwhile, numerical techniques designed for the single stochastic Schr\"{o}dinger or KdV equation are not directly transferable to the coupled setting, since the interaction terms substantially alter the dynamics of the two components and may break the structural properties preserved by each subsystem in isolation. Furthermore, deterministic structure-preserving numerical discretizations cannot, in general, be extended to the stochastic case in a straightforward manner, because the interplay between stochastic forcing and nonlinear coupling gives rise to new difficulties in preserving the relevant physical laws. As a result, the design of temporal and full discretizations that faithfully inherit the intrinsic structures of the continuous stochastic problem, including the  evolution laws of averaged mass, momentum, and  energy, remains a challenging and largely open problem. This situation calls for new numerical frameworks capable of simultaneously treating nonlinear coupling and stochastic perturbations while retaining the essential physical structures of the original system.

To bridge these gaps, we begin by deriving several fundamental evolution laws for averaged physical quantities of the continuous stochastic system, including the plasmon number, particle number, momentum, and energy. Building on these analytical results, we develop and analyze a new class of structure-preserving numerical methods for the SSKdV equation. Our approach is based on a splitting strategy that decouples the deterministic Hamiltonian part from the stochastic forcing. Within this setting, two Crank--Nicolson-type temporal discretizations are constructed. The first is tailored to preserve, at the discrete level, the evolution laws of the averaged plasmon number and momentum, together with the conservation law of the averaged particle number. To address the more intricate nonlinear energy structure, we further incorporate the constant scalar auxiliary variable (CSAV) technique \cite{Zhang2024CSAV}, which reformulates the energy contribution through an augmented system and thereby enables the construction of a scheme preserving the discrete averaged modified energy evolution law. Notably, the resulting CSAV-based temporal discretization appears to be new even for the deterministic Schr\"{o}dinger--KdV equation. In addition, it avoids the extra assumption, often imposed in standard SAV-type approaches, that the nonlinear functional be bounded from below. Coupling these temporal discretizations with the local discontinuous Galerkin (LDG) method \cite{Cockburn1998, Yan2015, xia2014conservative}, we obtain fully discretizations that inherit the corresponding physical evolution laws. Finally, numerical experiments are carried out to verify the theoretical analysis and to illustrate the accuracy and structure-preserving performance of the proposed methods for both the stochastic model and its deterministic counterpart.

The rest of the paper is organized as follows. In Section~\ref{sec:exact}, we introduce the SSKdV equation and discuss its intrinsic structural properties. Section~\ref{sec:semidiscrete} is devoted to the construction of temporal discretizations and the derivation of the corresponding discrete averaged evolution laws for the physical quantities. In Section~\ref{sec:fulldiscrete}, we further incorporate the LDG spatial discretization and show that the resulting fully discrete schemes inherit the associated discrete physical properties. Numerical experiments are reported in Section~\ref{sec:nmuexperiments} to verify the theoretical results and to illustrate the convergence behavior of the proposed methods. Finally, conclusions and possible directions for future research are presented in Section~\ref{sec:conclusion}. Before proceeding, we introduce some notation that will be used throughout the paper:
\begin{itemize}
\vspace{-0.7em}
\item $\partial_{x}^{k}$ denotes the partial derivative of order $k$ with respect to the spatial variable $x$, where $k =0,1,2,3,4$. We use the convention $\partial_x^0 = I$ and, for simplicity, write $\partial_x := \partial_x^1$;
\vspace{-0.7em}
\item $|\varphi|$ , $\overline{\varphi}$, $\Re(\varphi)$, and $\Im(\varphi)$ stand for the module, the complex conjugate, the real part, and the imaginary part of a complex-valued function $\varphi$, respectively;
\vspace{-0.7em}
\item $L^p(\mathcal O)$, $1\le p<\infty$, denotes the Banach space of complex-valued measurable functions defined on $\mathcal O$ whose $p$-th powers are integrable, equipped with the norm $\|\cdot\|_{L^p}$;
\vspace{-0.7em}
\item For $m \geq 0$, set $H^m := H^m(\mathcal O;\mathbb C),$ and denote by $H_{\mathbb R}^m:=H^m(\mathcal O;\mathbb R)
$ the corresponding real-valued Sobolev space. For simplicity, we use the same notation $\|\cdot\|_{H^m}$ for the norms on $H^m$, $H_{\mathbb R}^m$, and related product spaces whenever no confusion arises. 
Moreover, we write $H:=L^2(\mathcal O;\mathbb C)$ and $H_\mathbb R:=L^2(\mathcal O,\mathbb R)$. The space $H$ is equipped with the inner product $\langle \phi,\psi\rangle=\int_{\mathcal O}\phi(x)\overline{\psi(x)}\,\diff{x}$ for $\phi,\psi\in H$, and the norm $\|\cdot\|=\|\cdot\|_{L^2}$;

\vspace{-0.7em}
\item $\mathcal L_2(U_1,U_2)$ denotes the Hilbert space of Hilbert--Schmidt operators from $U_{1}$ to $U_{2}$.
\end{itemize}

      
          
          
          
          

\section{Physical properties of the SSKdV equation}\label{sec:exact}

In this section, we study the evolution law of physical quantities for the SSKdV equation with periodic boundary conditions
\begin{equation}\label{eq:SKeq}
\left\{\begin{aligned}
      &\diff{u}
      =
      \big(\i\partial_{x}^{2}u 
      - \i(\gamma_{1}v + \beta|u|^{2})u\big) \diff{t}
      + \Phi_{1}\diff{W_{1}(t)},
      \\
      &\diff{v} 
      = 
      \big(-\partial_{x}^{3}v
      + \gamma_{2}\partial_{x}|u|^{2} 
      - v\partial_{x}v\big)\diff{t} 
      + \Phi_{2}\diff{W_{2}(t)},
      \\
      &(u(0),v(0)) = (u_{0},v_{0}),\\
      &u(0, t)=u(L, t), \quad  u_x(0, t)=u_x(L, t), \\
      &v(0, t)=v(L, t),  \quad v_x(0, t)=v_x(L, t),\quad v_{xx}(0, t)=v_{xx}(L, t),
\end{aligned}\right.
\end{equation}
where $\mathcal{O} = (0,L)$ is the spatial domain with periodic boundary conditions, $\i = \sqrt{-1}$, and $\gamma_{1},\gamma_{2},\beta\in\mathbb{R}$. We assume that $\Phi_1\in \mathcal L_2(H_{\mathbb R};H^1)$, $\Phi_2\in \mathcal L_2(H_{\mathbb R};H_{\mathbb R}^1)$, and that $W_1$ and $W_2$ are two independent cylindrical Wiener processes on $H_{\mathbb R}$. More precisely, for $k=1,2$, 
$W_k(t)=\sum_{i=1}^{\infty}\beta_k^{(i)}(t)e_k^{(i)},$ 
where $\{\beta_k^{(i)}\}_{i\in\mathbb{N}}$ are independent real-valued standard Brownian motions, and $\{e_k^{(i)}\}_{i\in\mathbb{N}}$ are orthonormal bases of $H_{\mathbb R}$.
Assume that $u_{0} \in H^{1}$, $v_{0} \in H_{\R}^{1}$, $\gamma_{1}\gamma_{2} > 0$. Then system \eqref{eq:SKeq} admits a unique strong solution $(u,v)$ in $L^{2}\big(\Omega;C([0,T];H^{1}\times H^{1})\big)$ for $T>0$. This global well-posedness result follows from a priori estimates for $(u,v)$, and the details are analogous to those in \cite[Theorem 2.1]{chen2025global}.
It is straightforward to verify that the averaged particle number is conserved for the stochastic system \eqref{eq:SKeq}, namely,
\begin{align*}
\E\bigg[\int_{\mathcal{O}} v(t,x)\,\diff{x}\bigg] = \E\bigg[\int_{\mathcal{O}} v_{0}(x)\,\diff{x}\bigg], \quad t\in[0,T].
\end{align*}
In contrast, the plasmon number, which is conserved in the deterministic setting, satisfies an explicit evolution law in expectation. Applying the It\^o formula yields
\begin{align}\label{lem:continuousplasmon}
\E\bigg[\int_{\mathcal{O}}|u(t,x)|^{2}\diff{x}\bigg]
= \E\bigg[\int_{\mathcal{O}}|u_{0}(x)|^{2}\diff{x}\bigg]
+ \|\Phi_{1}\|_{\mathcal L_2^{0}}^{2}t, \quad t\in[0,T],
\end{align}
where $\|\Phi_{k}\|_{\mathcal L_2^{s}}^{2} := \sum_{i \in \mathbb{N}} \|\Phi_{k} e_{k}^{(i)}\|_{H^s}^2$ for $k = 1,2$ and $s \geq 0$.

Having established the expectation properties of the plasmon number and the particle number, we next consider the momentum functional for the stochastic system \eqref{eq:SKeq}. In contrast to the particle number, the momentum is generally not conserved in expectation, but satisfies an explicit linear evolution law.
\begin{lemma}
\label{lm;exact1}
      Assume that $u_{0} \in H^{1}$, $v_{0} \in H_{\R}^{1}$, $\gamma_{1}\gamma_{2} > 0$, $\Phi_{1} \in \mathcal{L}_{2}(H_{\R};H^{1})$, and $\Phi_{2} \in \mathcal{L}_{2}(H_{\R};H_{\R}^{1})$. Then for any $t \in [0,T]$, 
      \begin{align*}
            &~\E\bigg[\int_{\mathcal{O}} {\rm{Im}}\big(u(t,x)
            \overline{\partial_{x}u(t,x)}\big)
            + 
            \frac{\gamma_{1}}{2\gamma_{2}}
            \big(v(t,x)\big)^{2} \diff{x}\bigg] 
            \\=&~ \notag
            \E\bigg[\int_{\mathcal{O}}{\rm{Im}}\big(u_{0}
            \overline{\partial_{x}u_{0}}\big) 
            + 
            \frac{\gamma_{1}}{2\gamma_{2}} 
            \big(v_{0}\big)^{2} \diff{x}\bigg] 
            + 
            \Big({\rm{Im}}\big(
            {\rm{Tr}}(\Phi_{1}^{*}\partial_{x}\Phi_{1})\big)
            +
            \frac{\gamma_{1}}{2\gamma_{2}}
            \|\Phi_{2}\|_{\mathcal L_2^{0}}^{2}\Big)t,
      \end{align*}
 where $\partial_{x}\Phi_{1} := \partial_{x} \circ \Phi_{1}$. 
\end{lemma}
\begin{proof}
      Applying the It\^{o} formula leads to      
      \begin{align*}
           &~\diff \int_{\mathcal{O}} \big(u(t,x)
            \overline{\partial_{x}u(t,x)}\big)\diff{x}\\
            =&~
\left(\int_{\mathcal{O}}\left(\left(\i \partial_x^2 u\right) \overline{\partial_x u}-\left(\i \gamma_1 u v\right) \overline{\partial_x u}-\left(\i \beta|u|^2 u\right) \overline{\partial_x u}+u \overline{\i \partial_x^3 u}-u \overline{\i \gamma_1 \partial_x(u v)}-u \overline{\i\beta \partial_x\left(|u|^2 u\right)}\right) \diff{x}\right) \diff{t} 
\\ &~+\left(\int_{\mathcal{O}}\left(\Phi_1 \diff{W_1(t)} \overline{\partial_x u}+u \overline{\left(\partial_x \Phi_1\right) 
\diff{W_1(t)}}\right) \diff{x}\right)  
       +   \sum_{i=1}^\infty\int_{\mathcal{O}}
       \overline{\partial_x\Phi_1e_{1}^{(i)}}\Phi_1e_{1}^{(i)} \diff x\diff t .
      \end{align*}    
      By the boundary conditions and the integration by parts formula, one can verify that
      \begin{align*}
    &        \i\int_{\mathcal{O}} \partial_{x}^{2}u
            \overline{\partial_{x}u} 
            +
            \overline{\partial_{x}^{2}u}
            \partial_{x}u\diff{x}
            =
            \i\int_{\mathcal{O}} \partial_{x}
            |\partial_{x}u|^{2} \diff{x} 
            =
            0,\\
 &           \i\beta\int_{\mathcal{O}} |u|^{2}
            \big(u\overline{\partial_{x}u}
            +
            \overline{u} \partial_{x}u\big)\diff{x}  
            =
            \i\beta\int_{\mathcal{O}} |u|^{2}
            \partial_{x}|u|^{2} \diff{x}
            =
            0,
      \end{align*}
      as well as
      \begin{align*}
            \i\gamma_{1} \int_{\mathcal{O}} uv
            \overline{\partial_{x}u} \diff{x}
           + \i\gamma_{1}
            \int_{\mathcal{O}} \overline{u}v
            \partial_{x}u \diff{x}=&~
            \i\gamma_{1} \int_{\mathcal{O}} v\big(u
            \overline{\partial_{x}u}
            +
            \overline{u}\partial_{x}u\big) \diff{x}
            =
            \i\gamma_{1} \int_{\mathcal{O}} v
            \partial_{x}|u|^{2} \diff{x},
      \end{align*}      
      which implies 
      \begin{align}
      \label{eq:UpartialxU}
      &~\diff\int_{\mathcal{O}} {\rm{Im}}\big(u(t,x)
            \overline{\partial_{x}u(t,x)}\big)\diff{x}\nonumber\\
            =&~ {\rm{Im}}\bigg(
            \int_{\mathcal{O}}(\Phi_1 \diff{W_1(t)} \overline{\partial_x u}
            + u \overline{\partial_x \Phi_1 \diff{W_1(t)}} )\diff{x}
            +
      \sum_{i=1}^\infty\int_{\mathcal{O}}\overline{\partial_x\Phi_1e_{1}^{(i)}}\Phi_1e_{1}^{(i)} \diff{x}\diff{t}\bigg)\nonumber\\
      &~-
      \gamma_{1} {\rm{Re}}\bigg(\int_{\mathcal{O}} v
            \partial_{x}|u|^{2} \diff{x}\bigg)\diff t.         
       \end{align}       
      On the other hand, \eqref{eq:SKeq} further gives 
      \begin{align}\label{eq:VV}
            \diff{\langle v,v \rangle} \notag
            =&~
            \big(-2\langle \partial_{x}^{3}v,v \rangle
            +
            2\gamma_{2}\langle \partial_{x}|u|^{2},v \rangle
            -
            2\langle v\partial_{x}v,v \rangle \big)\diff{t} 
         + 
            2\langle \Phi_{2}\diff{W_{2}(t)},v \rangle \notag
            + 
            \|\Phi_{2}\|_{\mathcal L_2^{0}}^{2}\diff{t}
            \\=&~
            2\gamma_{2}\int_{\mathcal{O}}
            v\partial_{x}|u|^{2}\diff{x}\diff{t} 
            + 
            2\langle \Phi_{2}\diff{W_{2}(t)},v \rangle 
            + 
            \|\Phi_{2}\|_{\mathcal L_2^{0}}^{2}\diff{t},
      \end{align}
      where we have used $\langle v\partial_{x}v,v \rangle = 0$ and $\int_{\mathcal{O}} \partial_{x}v
            \partial_{x}^{2}v \diff{x}=0.$
      Combining \eqref{eq:UpartialxU} and \eqref{eq:VV}, and then taking expectations, gives the desired result.
\end{proof}

We now turn to the energy functional associated with the stochastic system \eqref{eq:SKeq}, whose evolution law is more delicate due to the nonlinear coupling terms.

\begin{lemma}
\label{thm:existunique}
      Assume that $u_{0} \in H^{1}$, $v_{0} \in H_{\R}^{1}$, $\gamma_{1}\gamma_{2} > 0$, $\Phi_{1} \in \mathcal{L}_{2}(H_{\R};H^{1})$, and $\Phi_{2} \in \mathcal{L}_{2}(H_{\R};H_{\R}^{1})$. Then for any $t \in [0,T]$, 
      \begin{align*}
            &~\E\bigg[\int_{\mathcal{O}} |\partial_{x}u(t)|^{2}
            + \frac{\beta}{2}|u(t)|^{4} + \gamma_{1}|u(t)|^{2}v(t)
            -
            \frac{\gamma_{1}}{6\gamma_{2}}(v(t))^{3} 
            + 
            \frac{\gamma_{1}}{2\gamma_{2}}|
            \partial_{x}v(t)|^{2} \diff{x}\bigg]
        \\=&~
            \E\bigg[\int_{\mathcal{O}} |\partial_{x}u_{0}|^{2} 
            +
            \frac{\beta}{2}|u_{0}|^{4}
            +
            \gamma_{1}|u_{0}|^{2}v_{0}
            -
            \frac{\gamma_{1}}{6\gamma_{2}}(v_{0})^{3} 
            +
            \frac{\gamma_{1}}{2\gamma_{2}}
            |\partial_{x}v_{0}|^{2} \diff{x}\bigg]
            \\&~+
            \bigg(\|\Phi_{1}\|_{\mathcal L_2^{1}}^{2}
            +
            \frac{\gamma_{1}}{2\gamma_{2}}
            \|\Phi_{2}\|_{\mathcal L_2^{1}}^{2}\bigg)t
            +
            \gamma_{1}\E\bigg[\int_{0}^{t}\int_{\mathcal{O}} 
             v\sum_{i=1}^{\infty} |\Phi_{1}e_{1}^{(i)}|^{2} 
             \diff{x}\diff{s} \bigg]
            \\&~+
            \beta\E\bigg[\int_{0}^{t} \int_{\mathcal{O}} \sum_{i=1}^{\infty}\Re\Big(u^{2}
            \overline{\Phi_{1}e_{1}^{(i)}}^{2}\Big)
            +
            2|u|^{2}|\Phi_{1}e_{1}^{(i)}|^{2} \diff{x}\diff{s}\bigg]            
            \\&~-
            \frac{\gamma_{1}}{2\gamma_{2}}
            \E\bigg[\int_{0}^{t}\int_{\mathcal{O}} v \sum_{i=1}^{\infty}
             (\Phi_{2}e_{2}^{(i)})^{2} \diff{x}\diff{s} \bigg].
      \end{align*} 
\end{lemma}

\begin{proof}
      By the It\^{o} formula to $\|\partial_x u\|^2$ and \eqref{eq:SKeq}, we arrive at
      \begin{align*}
            \diff{\|\partial_x u\|^2}
            =&~
            2\Re\big\langle \partial_x u,\i\partial_x^3 u 
            - \i\partial_x\big((\gamma_1v + \beta|u|^2)u\big)\big\rangle\diff{t}  
            +
            2\Re\big\langle \partial_x u,(\partial_x\Phi_1)\diff W_1(t)\big\rangle
            +
            \|\Phi_1\|_{\mathcal L_2^{0}}^2\diff t .
      \end{align*}
      Using integration by parts and the boundary conditions gives 
      $2\Re\big\langle \partial_x u,\i\partial_x^3 u \big\rangle = 0$ and
      \begin{align*}
            2\Re\big\langle \partial_{x} u,
            -\i\partial_{x}((\gamma_{1}v + \beta|u|^2)u)\big\rangle
            =&~
            2\gamma_1\Im\int_{\mathcal{O}}v\overline{u}\,\partial_{x}^{2}u\diff{x}
            +
            2\beta\Im\int_{\mathcal{O}}|u|^2\overline{u}\,\partial_{x}^{2}u\diff{x} .
      \end{align*}
      It follows that
      \begin{align}
            \E\bigg[\int_{\mathcal{O}}|\partial_{x}u|^{2}\diff{x}\bigg]
            =&~
            \E\bigg[\int_{\mathcal{O}}|\partial_{x}u^{0}|^{2}\diff{x}\bigg]
            +
            2\gamma_{1}\Im\E\bigg[\int_{0}^{t}\int_{\mathcal{O}}v\overline{u}
            \,\partial_{x}^{2}u \diff{x}\diff{s}\bigg]
            \notag\\
            &~
            +
            2\beta\Im\E\bigg[\int_{0}^{t}\int_{\mathcal{O}}|u|^{2}\overline{u}
            \,\partial_{x}^{2}u \diff{x}\diff{s}\bigg]
            + \|\Phi_{1}\|_{\mathcal{L}_{2}^{1}}^{2}t .
            \label{eq:exact1}
      \end{align}
      Applying the It\^{o} formula to $\|u\overline{u}\|^{2} = \int_{\mathcal{O}} |u|^{4} \diff{x}$ leads to
      \begin{align*}
            \diff{\|u\overline{u}\|^{2}}
            =&~
            4\Re\big\langle |u|^{2}u,\diff{u}\big\rangle
            +
            \bigg(2\sum_{i=1}^{\infty}\Re\int_{\mathcal{O}}
            u^{2}\overline{\Phi_{1}e_{1}^{(i)}}^{\,2}\diff{x}
            +
            4\sum_{i=1}^{\infty}\int_{\mathcal{O}}|u|^{2}
            |\Phi_{1}e_{1}^{(i)}|^{2}\diff{x}\bigg)\diff{t}.
      \end{align*}
      Since $4\Re\big\langle |u|^{2}u, \diff{u}\big\rangle
      = -4\Im\big\langle |u|^{2}u, \partial_{x}^{2}u \big\rangle\diff{t}
      +
      4\Re\big\langle |u|^{2}u, \Phi_{1}\diff{W_{1}(t)}\big\rangle$,
      we have
      \begin{align}
            \E\bigg[\int_{\mathcal{O}}\frac{\beta}{2}|u|^{4}\diff{x}\bigg]
            =&~
            \frac{\beta}{2}\E\bigg[\int_{\mathcal{O}}|u^{0}|^{4}\diff{x}\bigg]
            -2\beta\int_{0}^{t}\Im\langle |u|^{2}u, 
            \partial_{x}^{2}u \rangle\diff{s}
            \notag\\
            &~
            +
            \beta\E\bigg[\int_{0}^{t}\Re\sum_{i=1}^{\infty}\int_{\mathcal{O}}
            u^{2}\overline{\Phi_{1}e_{1}^{(i)}}^{\,2}
            +
            2|u|^{2}|\Phi_{1}e_{1}^{(i)}|^{2}\diff{x}\diff{s}\bigg].
            \label{eq:exact2}
      \end{align}
      Making use of the It\^{o} formula yields
      \[\diff{\int_{\mathcal{O}}|u|^{2}v\diff{x}}
      =
      \int_{\mathcal{O}}\big(2v\Re(\overline{u}\diff{u})
      + 
      |u|^{2}\diff{v} + v\diff{u}\diff{\overline{u}}\big)\diff{x}.\]
      Substituting \eqref{eq:SKeq} and using integration by parts, we derive
      \begin{align*}
            \diff{\int_{\mathcal{O}} |u|^{2}v \diff{x}}
            =&~
            2\Re\int_{\mathcal{O}}\i v\overline{u}\,\partial_{x}^{2}u\diff{x}\diff{t}
            +
            2\Re\big\langle v\overline{u}, \Phi_{1}\diff{W_{1}}\big\rangle 
            +
            \int_{\mathcal{O}} v\sum_{i=1}^{\infty}
            |\Phi_{1}e_{1}^{(i)}|^{2}\diff{x}\diff{t} \\
            &~
            +
            \int_{\mathcal{O}}\partial_{x}|u|^{2}
            \,\partial_{x}^{2}v \diff{x}\diff{t} 
            +
            \frac{1}{2}\int_{\mathcal{O}} v^{2}\partial_{x}|u|^{2}\diff{x}\diff{t}
            +
            \big\langle |u|^{2},\Phi_{2}\diff{W_2} \big\rangle,
      \end{align*}
      which implies
      \begin{align}
            \E\bigg[\int_{\mathcal{O}}\gamma_{1}|u|^{2}v\diff{x}\bigg]
            =&~
            \E\bigg[\int_{\mathcal{O}}\gamma_{1}|u^{0}|^{2} v^{0}\diff{x}\bigg]
            +
            2\gamma_{1}\Re\E\bigg[\int_{0}^{t}\int_{\mathcal{O}}\i v\overline{u}
            \,\partial_{x}^{2}u \diff{x}\diff{s}\bigg]
            \notag\\
            &~
            +
            \gamma_{1}\E\bigg[\int_{0}^{t}\int_{\mathcal{O}}v\sum_{i=1}^{\infty}
            |\Phi_{1}e_{1}^{(i)}|^{2} \diff{x}\diff{s}\bigg]
            +
            \gamma_{1}\E\bigg[\int_{0}^{t}\int_{\mathcal{O}}\partial_{x}|u|^{2}
            \,\partial_{x}^{2} v\diff{x}\diff{s}\bigg]
            \notag\\
            &~
            +\frac{\gamma_1}{2}\E\bigg[\int_{0}^{t}\int_{\mathcal{O}}
            v^{2}\partial_{x}|u|^{2}\diff{x}\diff{s}\bigg].
            \label{eq:exact3}
      \end{align}
      With the help of the It\^{o} formula, we deduce
      \begin{align*}
            \diff{\int_{\mathcal{O}} v^{3} \diff{x}}
            =&~
            3\int_{\mathcal{O}} \partial_{x} v^{2}\,\partial_{x}^{2} v\diff{x}\diff{t}
            +
            3\gamma_{2}\int_{\mathcal{O}} v^{2}\partial_{x}|u|^{2} \diff{x}\diff{t} 
            \\&~
            +
            3\big\langle v^{2},\Phi_{2}\diff{W_{2}} \big\rangle
            +
            3\int_{\mathcal{O}} v\sum_{i=1}^{\infty}(\Phi_{2}e_{2}^{(i)})^{2}
            \diff{x}\diff{t}.
      \end{align*}
      Consequently,
      \begin{align}
            \E\bigg[\int_{\mathcal{O}} -\frac{\gamma_{1}}{6\gamma_{2}}v^{3}\diff{x}\bigg]
            =&~
            \E\bigg[\int_{\mathcal{O}} -\frac{\gamma_{1}}{6\gamma_{2}}(v^{0})^{3}\diff{x}\bigg]
            -
            \frac{\gamma_{1}}{2\gamma_{2}}\E\bigg[\int_{0}^{t}\int_{\mathcal{O}}
            \partial_{x} v^{2}\,\partial_{x}^{2} v\diff{x}\diff{s}\bigg]
            \notag\\
            &~
            -
            \frac{\gamma_{1}}{2}\E\bigg[\int_{0}^{t}\int_{\mathcal{O}} 
            v^{2}\partial_{x}|u|^{2}\diff{x}\diff{s}\bigg]
            -
            \frac{\gamma_{1}}{2\gamma_{2}}\E\bigg[\int_{0}^{t}\int_{\mathcal{O}}
            v\sum_{i=1}^{\infty}(\Phi_{2}e_{2}^{(i)})^{2}\diff{x}\diff{s}\bigg].
            \label{eq:exact4}
      \end{align}
      Finally, based on the It\^{o} formula to $\|\partial_xV\|^2$ and integration by parts, we get
      \begin{align*}
            \diff{\|\partial_{x}v\|^{2}}
            =&~
            -2\gamma_{2}\int_{\mathcal{O}}\partial_{x}|u|^{2}
            \,\partial_{x}^{2}v\diff{x}\diff{t}
            +
            \int_{\mathcal{O}}\partial_{x}v^{2}\,\partial_{x}^{2} v\diff{x}\diff{t} 
            +
            2\big\langle \partial_{x}v,(\partial_{x}\Phi_{2})\diff{W_2}\big\rangle
            +
            \|\Phi_{2}\|_{\mathcal{L}_2^{1}}^2\diff{t},
      \end{align*}
      and accordingly
      \begin{align}
            \E\bigg[\int_{\mathcal{O}}\frac{\gamma_{1}}{2\gamma_{2}}
            |\partial_{x}v|^{2}\diff{x}\bigg]
            =&~
            \E\bigg[\int_{\mathcal{O}}\frac{\gamma_{1}}{2\gamma_{2}}
            |\partial_{x}v^{0}|^{2}\diff{x}\bigg]
            -
            \gamma_{1}\E\bigg[\int_{0}^{t}\int_{\mathcal{O}}
            \partial_{x}|u|^{2}\,\partial_{x}^{2}v\diff{x}\diff{s}\bigg]
            \notag\\
            &~
            +
            \frac{\gamma_{1}}{2\gamma_{2}}\E\bigg[\int_{0}^{t}\int_{\mathcal{O}}
            \partial_{x}v^{2}\,\partial_{x}^{2}v\diff{x}\diff{s}\bigg]
            +
            \frac{\gamma_{1}}{2\gamma_{2}}\|\Phi_{2}\|_{\mathcal{L}_{2}^{1}}^{2}t .
      \label{eq:exact5}
      \end{align}
      Combining \eqref{eq:exact1}--\eqref{eq:exact5}, we obtain the desired result.
\end{proof}

\section{Structure-preserving temporal discretizations}\label{sec:semidiscrete}
In this section, we present temporal discretizations that inherit the physical properties of the SSKdV equation \eqref{eq:SKeq}. To this end, we first employ a splitting framework that decompose \eqref{eq:SKeq} into the deterministic conservative system \eqref{eq:detsubsystem} and the stochastic subsystem
\begin{equation}\label{eq:stosubsystem}
      \diff{u} = \Phi_{1}\diff{W_{1}(t)},
      \quad 
      \diff{v} = \Phi_{2}\diff{W_{2}(t)}.
\end{equation}
This decomposition decouples the nonlinear conservative deterministic dynamics from the stochastic forcing, and allows each part to be treated by a numerical strategy adapted to its intrinsic structure.

Let $N \in \N$, $\Delta t = \frac{T}{N}$, and define $t_n = n\Delta t$ for $n \in Z_N:=\{0,1,\cdots,N\}$. For $i \in\{1, 2\}$, we denote the increments of the cylindrical Wiener processes by 
$$\Delta W_{k}^{n} := W_{k}(t_{n+1}) - W_{k}(t_n), \quad k = 1,2, n = 0,1,\cdots,N-1.$$ 
Our first temporal discretization is constructed by combining a deterministic--stochastic splitting strategy, the Crank--Nicolson midpoint rule for the deterministic subsystem, and the exact flow of the stochastic subsystem \eqref{eq:stosubsystem}. More precisely, given $(U^0,V^0) = (u^{0},v^{0})$ and an approximation $(U^n,V^n)$ of $(u(t_n),v(t_n))$, we first apply the Crank--Nicolson midpoint discretization to the deterministic subsystem and compute the intermediate variables $(\widetilde U^{n+1},\widetilde V^{n+1})$ through
\begin{equation}\label{eq:temproaldis1}
\left\{\begin{aligned}
&\frac{\widetilde U^{n+1}-U^n}{\Delta t}
=
\i\partial_x^2 U^{n+\frac12}
-
\i\Big(
\gamma_1 U^{n+\frac12}V^{n+\frac12}
+
\beta |U^{n+\frac12}|^2 U^{n+\frac12}
\Big),
\\
&\frac{\widetilde V^{n+1}-V^n}{\Delta t}
=
-\partial_x^3 V^{n+\frac12}
+
\gamma_2\partial_x |U^{n+\frac12}|^2
-
V^{n+\frac12}\partial_x V^{n+\frac12},
\end{aligned}\right.
\end{equation}
where
\begin{equation}\label{eq:temproaldis1UVmid}
U^{n+\frac12}
:=
\frac{U^n+\widetilde U^{n+1}}{2},
\quad
V^{n+\frac12}
:=
\frac{V^n+\widetilde V^{n+1}}{2}.
\end{equation}
After evolving the deterministic part, we incorporate the stochastic contribution by using the exact solution of \eqref{eq:stosubsystem}, namely,
\begin{equation}\label{eq:temproaldis1UVn1}
U^{n+1}
=
\widetilde U^{n+1}
+
\Phi_1\Delta W_1^n,
\quad
V^{n+1}
=
\widetilde V^{n+1}
+
\Phi_2\Delta W_2^n.
\end{equation}
The Crank--Nicolson midpoint rule is particularly useful for preserving the intrinsic conservative structure of the deterministic subsystem. Combined with the exact treatment of the stochastic perturbations, the resulting splitting scheme provides a robust temporal discretization for the original stochastic system and exhibits favorable structure-preserving properties.

Below, we derive several averaged evolution laws associated with the temporal discretization \eqref{eq:temproaldis1}--\eqref{eq:temproaldis1UVn1}.

\begin{lemma}
Under the assumptions of Lemma~\ref{lm;exact1}, the discretization \eqref{eq:temproaldis1}--\eqref{eq:temproaldis1UVn1} preserves the following discrete properties, that is, for all $n\in Z_{N}$,
\begin{align*}
&~\E\bigg[\int_{\mathcal O}|U^n(x)|^2\diff x\bigg]
=
\E\bigg[\int_{\mathcal O}|U^0(x)|^2\diff x\bigg]
+
\|\Phi_1\|_{\mathcal L_2^{0}}^2 t_n,
\\
&~\E\bigg[\int_{\mathcal O}V^n(x)\diff x\bigg]
=
\E\bigg[\int_{\mathcal O}U^0(x)\diff x\bigg],
\\
&~\E\bigg[\int_{\mathcal O}{\rm Im}\big(U^n\overline{\partial_xU^n}\big)
+
\frac{\gamma_1}{2\gamma_2}(V^n)^2\diff x\bigg]\\
=&~
\E\bigg[\int_{\mathcal O}{\rm Im}\big(U^0\overline{\partial_xU^0}\big)
+
\frac{\gamma_1}{2\gamma_2}(U^0)^2\diff x\bigg]
+
\Big(
{\rm Im}\big({\rm Tr}(\Phi_1^*\partial_x\Phi_1)\big)
+
\frac{\gamma_1}{2\gamma_2}\|\Phi_2\|_{\mathcal L_2^{0}}^2
\Big)t_n.
\end{align*}
\end{lemma}

\begin{proof}
Using the polarization identity
$
|a|^2-|b|^2
=
\Re\big((a-b)\overline{(a+b)}\big)$, 
\eqref{eq:temproaldis1UVmid} and integration by parts, we obtain
\begin{align*}
\frac{1}{2\Delta t}
\Big(
\|\widetilde U^{n+1}\|^2-\|U^n\|^2
\Big)
&=
\Re\bigg\langle
\frac{\widetilde U^{n+1}-U^n}{\Delta t},
U^{n+\frac12}
\bigg\rangle=0, \quad \int_{\mathcal O}\widetilde V^{n+1}\diff x
=
\int_{\mathcal O}V^n\diff x.
\end{align*}
Define the discrete momentum functional
$$\mathcal M(U,V)
:=
\int_{\mathcal O}
{\rm Im}(U\overline{\partial_xU})
+
\frac{\gamma_1}{2\gamma_2}V^2
\diff x.
$$
By \eqref{eq:temproaldis1} and the identity
\[
\widetilde U^{n+1}\overline{\partial_x\widetilde U^{n+1}}
-
U^n\overline{\partial_xU^n}
=
(\widetilde U^{n+1}-U^n)\overline{\partial_xU^{n+\frac12}}
+
U^{n+\frac12}
\overline{\partial_x(\widetilde U^{n+1}-U^n)},
\]
a direct calculation combined with integration by parts gives 
$\mathcal M(\widetilde U^{n+1},\widetilde V^{n+1})
=
\mathcal M(U^n,V^n).$

We now turn to the stochastic update \eqref{eq:temproaldis1UVn1}. 
Since the Wiener increments have zero mean and covariance $\Delta t$, taking expectations gives
\begin{align*}
\E\big[\|U^{n+1}\|^2\big]
&=
\E\big[\|\widetilde U^{n+1}\|^2\big]
+
\|\Phi_1\|_{\mathcal L_2^{0}}^2\Delta t,\quad 
\E\bigg[\int_{\mathcal O}V^{n+1}\diff x\bigg]
=
\E\bigg[\int_{\mathcal O}\widetilde V^{n+1}\diff x\bigg],
\end{align*}
and
\begin{align*}
\E\big[\mathcal M(U^{n+1},V^{n+1})\big]
&=
\E\big[\mathcal M(\widetilde U^{n+1},\widetilde V^{n+1})\big]
+
\Big(
{\rm Im}\big({\rm Tr}(\Phi_1^*\partial_x\Phi_1)\big)
+
\frac{\gamma_1}{2\gamma_2}\|\Phi_2\|_{\mathcal L_2^{0}}^2
\Big)\Delta t.
\end{align*}
Combining the above relations with the conservation properties of the deterministic substep and iterating over $n$ completes the proof.
\end{proof}


We have therefore derived the desired  evolution laws for the averaged plasmon number, particle number, and momentum at the temporally semi-discrete level. The situation is, however, different for the energy. Unlike the above quantities, the original energy functional is not preserved by the temporal discretization \eqref{eq:temproaldis1}--\eqref{eq:temproaldis1UVn1}. More precisely, while the Crank--Nicolson midpoint approximation maintains the skew-symmetric cancellation mechanism underlying the momentum functional, it fails to exactly preserve the nonlinear variational structure associated with the energy. As a consequence, the intermediate state 
$(\widetilde U^{n+1},\widetilde V^{n+1})$ produced by the deterministic step \eqref{eq:temproaldis1} and \eqref{eq:temproaldis1UVmid} no longer satisfies the energy conservation law of the deterministic system \eqref{eq:detsubsystem}. It follows that the evolution law for the averaged energy cannot be recovered directly from the first temporal discretization.

To overcome this difficulty, we introduce two CSAVs. The key idea is to reformulate the nonlinear energy contributions through two auxiliary scalar variables so that the modified system admits a tractable discrete energy structure. 
We first recall that the deterministic subsystem \eqref{eq:detsubsystem} satisfies the following energy conservation law
\begin{align*}
E(t)
:=
\int_{\mathcal O}
|\partial_xU(t)|^2
+
\frac{\beta}{2}|U(t)|^4
+
\gamma_1|U(t)|^2V(t)
-
\frac{\gamma_1}{6\gamma_2}V(t)^3
+
\frac{\gamma_1}{2\gamma_2}|\partial_xV(t)|^2
\diff x
=
E(0),
\quad t\in[0,T].
\end{align*}
To implement the CSAV reformulation, we define
\begin{align*}
E_1(t)
&:=
\int_{\mathcal O}
\bigg(
\frac{\beta}{2}|U(t)|^4
+
\gamma_1|U(t)|^2V(t)
\bigg)\diff x,\quad 
E_2(t):=
-\frac{\gamma_1}{6\gamma_2}
\int_{\mathcal O}V(t)^3\diff x,
\quad t\in[0,T],
\end{align*}
and introduce two  scalar auxiliary variables $r(t)$ and $s(t)$ satisfying
\begin{equation*}
\left\{
\begin{aligned}
&\frac{\diff r(t)}{\diff t}
=
\alpha_1
\bigg(
-\frac{\diff E_1(t)}{\diff t}
+
r(t)P_1(t)
\bigg),\quad r(0)=1,
\\
&\frac{\diff s(t)}{\diff t}
=
\alpha_2
\bigg(
-\frac{\diff E_2(t)}{\diff t}
+
s(t)P_2(t)
\bigg),\quad s(0)=1,
\end{aligned}
\right.
\end{equation*}
where $\alpha_{1}, \alpha_{2} > 0$ are prescribed CSAV parameters, 
$P_2(t) := -\frac{\gamma_1}{2\gamma_2} \int_{\mathcal O}V^2\partial_tV\diff x,$ and
\begin{align*}
P_1(t)
:=&~
\,{\rm Re}\bigg(2\beta
\int_{\mathcal O}|U|^2U\overline{\partial_tU}\diff{x}
+
2\gamma_1
\int_{\mathcal O}UV\overline{\partial_tU}\diff{x}\bigg)
+
\gamma_1
\int_{\mathcal O}|U|^2\partial_tV\diff x.
\end{align*}
Since $P_{1}(t) = \frac{\diff{E_{1}(t)}}{\diff{t}}$ and $P_{2}(t) = \frac{\diff{E_{2}(t)}}{\diff{t}}$, the above auxiliary equations imply that $r(t) = s(t) = 1$ for all $t \geq 0$. Therefore, we obtain the following modified system
\begin{equation}\label{eq:detsubmodif}
\left\{
\begin{aligned}
&\partial_tU
=
\i\partial_x^2U
-
\i r(t)
\big(
\gamma_1UV+\beta|U|^2U
\big),
\\
&\partial_tV
=
-\partial_x^3V
+
\gamma_2 r(t)\partial_x|U|^2
-
s(t)V\partial_xV,
\end{aligned}
\right.
\end{equation}
which is equivalent to the subsystem \eqref{eq:detsubsystem}.
Moreover, one can verify that the modified energy functional
\begin{align*}
\widetilde E(t)
:=
\int_{\mathcal O}
|\partial_xU(t)|^2
+
\frac{\beta}{2}|U(t)|^4
+
\gamma_1|U(t)|^2V(t)
-
\frac{\gamma_1}{6\gamma_2}V(t)^3
+
\frac{\gamma_1}{2\gamma_2}
|\partial_xV(t)|^2
\diff x
+
\frac{r(t)}{\alpha_1}
+
\frac{s(t)}{\alpha_2}
\end{align*}
is conserved, namely, $\widetilde E(t) = \widetilde E(0)$ for all $t \in [0,T]$.

We now construct the second temporal discretization based on the modified system \eqref{eq:detsubmodif}. Let $(U^{n},V^{n}, r^{n}, s^{n})$ be the numerical approximation to $(u(t_{n}),v(t_{n}), r(t_{n}), s(t_{n}))$. Correspondingly, we define the discrete nonlinear energy components
\begin{align*}
E_{1}^{n}
:=
\int_{\mathcal O} \Big(\frac{\beta}{2}|U^{n}|^4
+ \gamma_1 |U^{n}|^{2} V^{n}\Big)\diff{x},
\quad
E_{2}^{n}
:=
-\frac{\gamma_1}{6\gamma_2}
\int_{\mathcal O} (V^{n})^{3} \diff{x}.
\end{align*}
The evolution from time level $t_n$ to $t_{n+1}$ is carried out through intermediate variables $\widetilde U^{n+1}$ and $\widetilde V^{n+1}$, where $\widetilde U^{n+1}$ satisfies
\begin{equation}\label{eq:timedis2U}
\frac{\widetilde{U}^{n+1}-U^{n}}{\Delta{t}}
=
\i\partial_{x}^{2}U^{n+\frac{1}{2}}
-
\i r^{n+\frac{1}{2}} \big(\gamma_{1}U^{n+\frac{1}{2}}
V^{n+\frac{1}{2}}
+
\beta|U^{n+\frac{1}{2}}|^{2}
U^{n+\frac{1}{2}}\big)
\end{equation}
with
\begin{align*}
U^{n+\frac{1}{2}} := \frac{U^{n} + \widetilde{U}^{n+1}}{2},\quad
V^{n+\frac{1}{2}} := \frac{V^{n} + \widetilde{V}^{n+1}}{2},\quad
r^{n+\frac{1}{2}} := \frac{r^{n} + r^{n+1}}{2}.
\end{align*}
In a similar manner, the intermediate variable $\widetilde V^{n+1}$ is obtained from
\begin{equation}\label{eq:timedis2V}
\frac{\widetilde{V}^{n+1}-V^{n}}{\Delta{t}}
=
-\partial_{x}^{3}V^{n+\frac{1}{2}}
+
\gamma_{2}r^{n+\frac{1}{2}}
\partial_{x}|U^{n+\frac{1}{2}}|^{2}
-
s^{n+\frac{1}{2}}V^{n+\frac{1}{2}}\partial_{x}V^{n+\frac{1}{2}}
\end{equation}
with
$
s^{n+\frac{1}{2}} := \frac{s^{n} + s^{n+1}}{2}.
$
Here, $r^{n+1}$ is determined through
\begin{align}\label{eq:timedis2r}
&\frac{r^{n+1}-r^{n}}{\Delta{t}}\\
=&~
\alpha_{1} \bigg( -\frac{1}{\Delta{t}}
\bigg(\int_{\mathcal O} \frac{\beta}{2}|\widetilde{U}^{n+1}|^{4}
+ \gamma_1 |\widetilde{U}^{n+1}|^{2} \widetilde{V}^{n+1}\diff{x}
-
\int_{\mathcal O} \frac{\beta}{2}|U^{n}|^{4}
+ \gamma_1 |U^{n}|^{2} V^{n} \diff{x}\bigg)
+ r^{n+\frac{1}{2}} P_1^{n+\frac{1}{2}} \bigg),\nonumber
\end{align}
where
\begin{align*}
P_{1}^{n+\frac{1}{2}}
&=
2\beta \Re\int_{\mathcal O} |U^{n+\frac{1}{2}}|^{2} U^{n+\frac{1}{2}}
\overline{\frac{\widetilde{U}^{n+1}-U^{n}}{\Delta{t}}} \diff{x}
+ 2\gamma_{1} \Re\int_{\mathcal O} U^{n+\frac{1}{2}} V^{n+\frac{1}{2}}
\overline{\frac{\widetilde{U}^{n+1}-U^{n}}{\Delta{t}}} \diff{x}\\
&\quad
+\gamma_{1} \int_{\mathcal O} |U^{n+\frac{1}{2}}|^2
\frac{\widetilde{V}^{n+1}-V^{n}}{\Delta{t}} \diff{x}.
\end{align*}
Analogously, the evolution of $s^{n+1}$ is governed by
\begin{equation}\label{eq:timedis2s}
\frac{s^{n+1}-s^{n}}{\Delta{t}}
= \alpha_2 \bigg( -\frac{-\frac{\gamma_1}{6\gamma_2}
\int_{\mathcal O} (\widetilde{V}^{n+1})^{3} - (V^{n})^{3} \diff{x}}{\Delta{t}}
+ s^{n+\frac{1}{2}} P_2^{n+\frac{1}{2}} \bigg),
\end{equation}
where
$
P_2^{n+\frac{1}{2}}
=
-\frac{\gamma_1}{2\gamma_2} \int_{\mathcal O} (V^{n+\frac{1}{2}})^{2}
\frac{\widetilde{V}^{n+1}-V^{n}}{\Delta{t}} \diff{x}.$
The variables $\widetilde U^{n+1}$, $\widetilde V^{n+1}$, $r^{n+1}$, and $s^{n+1}$ are thus determined implicitly and simultaneously. Based on the exact solution of the stochastic linear subsystem, we finally define the CSAV-based temporal discretization by
\begin{equation}\label{eq:timedis2UV}
U^{n+1} := \widetilde{U}^{n+1} + \Phi_{1}\Delta{W}_{1}^{n},
\quad
V^{n+1} := \widetilde{V}^{n+1} + \Phi_{2}\Delta{W}_{2}^{n}.
\end{equation}
To the best of our knowledge, this CSAV-based scheme \eqref{eq:timedis2U}--\eqref{eq:timedis2UV} appears to be new even for the corresponding deterministic Schr\"{o}dinger--KdV system. We next establish its averaged energy evolution law.

\begin{lemma}
Under the assumptions of Lemma \ref{thm:existunique}, for any $n \in Z_{N}$, the CSAV-based temporal discretization \eqref{eq:timedis2U}--\eqref{eq:timedis2UV} satisfies
\begin{align*}
&~\E\bigg[\int_{\mathcal O} |\partial_xU^n|^2+\frac{\beta}{2}|U^n|^4
+\gamma_1|U^n|^2V^n-\frac{\gamma_1}{6\gamma_2}(V^n)^3
+\frac{\gamma_1}{2\gamma_2}|\partial_xV^n|^2 \diff x\bigg]
\\
=&~\E\bigg[\int_{\mathcal O} |\partial_xU^0|^2+\frac{\beta}{2}|U^0|^4
+\gamma_1|U^0|^2V^0-\frac{\gamma_1}{6\gamma_2}(V^0)^3
+\frac{\gamma_1}{2\gamma_2}|\partial_xV^0|^2 \diff x\bigg]
\\
&~+\bigg(\|\Phi_1\|_{\mathcal L_2^{1}}^2
+\frac{\gamma_1}{2\gamma_2}\|\Phi_2\|_{\mathcal L_2^{1}}^2\bigg)t_n
+\gamma_1\sum_{j=0}^{n-1}\E\bigg[\int_{\mathcal O}V^{j+1}
\sum_{i=1}^\infty|\Phi_1e_1^{(i)}|^2\diff x\bigg]\Delta t
\\
&~-\frac{\gamma_1}{2\gamma_2}\sum_{j=0}^{n-1}
\E\bigg[\int_{\mathcal O}V^{j+1}
\sum_{i=1}^\infty(\Phi_2e_2^{(i)})^2\diff x\bigg]\Delta t
+\frac{r^0-r^n}{\alpha_1}+\frac{s^0-s^n}{\alpha_2}
\\
&~+\beta\Delta t\sum_{j=0}^{n-1}\E\bigg[\int_{\mathcal O}\sum_{i=1}^\infty
{\rm Re}\Big((U^{j+1})^2\overline{\Phi_1e_1^{(i)}}^{\,2}\Big)
+2|U^{j+1}|^2|\Phi_1e_1^{(i)}|^2\diff x\bigg].
\end{align*}
\end{lemma}

\begin{proof}
We first prove the pathwise modified energy conservation law
\begin{align}\label{eq:detsubenergycon_new}
&~\int_{\mathcal O} |\partial_x\widetilde U^{n+1}|^2
+\frac{\beta}{2}|\widetilde U^{n+1}|^4
+\gamma_1|\widetilde U^{n+1}|^2\widetilde V^{n+1}
-\frac{\gamma_1}{6\gamma_2}(\widetilde V^{n+1})^3
+\frac{\gamma_1}{2\gamma_2}|\partial_x\widetilde V^{n+1}|^2 \diff x
+\frac{r^{n+1}}{\alpha_1}+\frac{s^{n+1}}{\alpha_2}
\notag
\\
=&~\int_{\mathcal O} |\partial_xU^n|^2
+\frac{\beta}{2}|U^n|^4+\gamma_1|U^n|^2V^n
-\frac{\gamma_1}{6\gamma_2}(V^n)^3
+\frac{\gamma_1}{2\gamma_2}|\partial_xV^n|^2 \diff x
+\frac{r^n}{\alpha_1}+\frac{s^n}{\alpha_2}.
\end{align} 
Indeed, by integration by parts and \eqref{eq:timedis2U},
\begin{align*}
\int_{\mathcal O}\big(|\partial_x\widetilde U^{n+1}|^2-|\partial_xU^n|^2\big)\diff x
=&~-2\Re\int_{\mathcal O}(\overline{\widetilde U^{n+1}}-\overline U^n)
\partial_x^2U^{n+\frac12}\diff x
\\
=&~-2r^{n+\frac12}\Re\int_{\mathcal O}
(\overline{\widetilde U^{n+1}}-\overline U^n)
\big(\gamma_1V^{n+\frac12}+\beta|U^{n+\frac12}|^2\big)
U^{n+\frac12}\diff x .
\end{align*}
Similarly, using \eqref{eq:timedis2V} yields
\begin{align*}
&~\frac{\gamma_1}{2\gamma_2}\int_{\mathcal O}
\big(|\partial_x\widetilde V^{n+1}|^2-|\partial_xV^n|^2\big)\diff x
\\=&~
-\frac{\gamma_1}{\gamma_2}\Delta t
\int_{\mathcal O}
\Big(\gamma_2r^{n+\frac12}\partial_x|U^{n+\frac12}|^2
-s^{n+\frac12}V^{n+\frac12}\partial_xV^{n+\frac12}\Big)
\partial_x^2V^{n+\frac12}\diff x .
\end{align*}
Moreover, \eqref{eq:timedis2r}--\eqref{eq:timedis2s} show
\begin{align*}
\frac{r^{n+1}-r^n}{\alpha_1}
=&~-\int_{\mathcal O}\Big(
\frac{\beta}{2}|\widetilde U^{n+1}|^4
+\gamma_1|\widetilde U^{n+1}|^2\widetilde V^{n+1}
-\frac{\beta}{2}|U^n|^4-\gamma_1|U^n|^2V^n
\Big)\diff x+\Delta tr^{n+\frac12}P_1^{n+\frac12},
\\
\frac{s^{n+1}-s^n}{\alpha_2}
=&~\frac{\gamma_1}{6\gamma_2}
\int_{\mathcal O}\big((\widetilde V^{n+1})^3-(V^n)^3\big)\diff x
+\Delta ts^{n+\frac12}P_2^{n+\frac12}.
\end{align*}
It follows from the above identities that
\begin{align*}
Error
:=&~\eqref{eq:detsubenergycon_new}\text{(LHS)}
-\eqref{eq:detsubenergycon_new}\text{(RHS)}
\\
=&~-\frac{\gamma_1}{\gamma_2}\Delta t
\int_{\mathcal O}
\Big(\gamma_2r^{n+\frac12}\partial_x|U^{n+\frac12}|^2
-s^{n+\frac12}V^{n+\frac12}\partial_xV^{n+\frac12}\Big)
\partial_x^2V^{n+\frac12}\diff x
\\
&~+\gamma_1r^{n+\frac12}
\int_{\mathcal O}|U^{n+\frac12}|^2
(\widetilde V^{n+1}-V^n)\diff x-\frac{\gamma_1}{2\gamma_2}s^{n+\frac12}
\int_{\mathcal O}(V^{n+\frac12})^2
(\widetilde V^{n+1}-V^n)\diff x .
\end{align*}
Integrating by parts gives
\begin{align*}
Error
=&~\gamma_1r^{n+\frac12}\Delta t
\int_{\mathcal O}|U^{n+\frac12}|^2
\partial_x^3V^{n+\frac12}\diff x-\frac{\gamma_1}{2\gamma_2}s^{n+\frac12}\Delta t
\int_{\mathcal O}(V^{n+\frac12})^2
\partial_x^3V^{n+\frac12}\diff x .
\end{align*}
Since \eqref{eq:timedis2V} implies
\[
\partial_x^3V^{n+\frac12}
=-\frac{\widetilde V^{n+1}-V^n}{\Delta t}
+\gamma_2r^{n+\frac12}\partial_x|U^{n+\frac12}|^2
-s^{n+\frac12}V^{n+\frac12}\partial_xV^{n+\frac12},
\]
we obtain $Error=0$, and hence \eqref{eq:detsubenergycon_new} follows. 
Next, using \eqref{eq:timedis2UV} and It\^o isometry leads to
\begin{align}
\E\big[\|\partial_xU^{n+1}\|^2\big]
=&~\E\big[\|\partial_x\widetilde U^{n+1}\|^2\big]
+\|\Phi_1\|_{\mathcal L_2^{1}}^2\Delta t,
\label{eq:energy_aux1}
\\
\E\bigg[\frac{\gamma_1}{2\gamma_2}\|\partial_xV^{n+1}\|^2\bigg]
=&~\E\bigg[\frac{\gamma_1}{2\gamma_2}
\|\partial_x\widetilde V^{n+1}\|^2\bigg]
+\frac{\gamma_1}{2\gamma_2}
\|\Phi_2\|_{\mathcal L_2^{1}}^2\Delta t .
\label{eq:energy_aux2}
\end{align}
Similarly,  
\begin{align}
&~\E\bigg[\int_{\mathcal O}\gamma_1|U^{n+1}|^2V^{n+1}\diff x\bigg]
=\E\bigg[\int_{\mathcal O}
\gamma_1|\widetilde U^{n+1}|^2\widetilde V^{n+1}\diff x
+\gamma_1\Delta t\int_{\mathcal O}V^{n+1}
\sum_{i=1}^\infty|\Phi_1e_1^{(i)}|^2\diff x\bigg],
\label{eq:energy_aux3}\\
&~\E\bigg[\int_{\mathcal O}
-\frac{\gamma_1}{6\gamma_2}(V^{n+1})^3\diff x\bigg]
=\E\bigg[\int_{\mathcal O}
-\frac{\gamma_1}{6\gamma_2}(\widetilde V^{n+1})^3\diff x
-\frac{\gamma_1}{2\gamma_2}\Delta t
\int_{\mathcal O}V^{n+1}
\sum_{i=1}^\infty(\Phi_2e_2^{(i)})^2\diff x\bigg].
\label{eq:energy_aux4}
\end{align}
For the quartic term, expanding
$|U^{n+1}|^2
=
|\widetilde U^{n+1}|^2
+ |\Phi_1\Delta W_1^n|^2
+ 2{\rm Re}(\widetilde U^{n+1}\overline{\Phi_1\Delta W_1^n})$
and using
\[
\E\big[|\Phi_1\Delta W_1^n|^2\big]
=\Delta t\sum_{i=1}^\infty|\Phi_1e_1^{(i)}|^2,
\quad
\E\big[(\Phi_1\Delta W_1^n)^2\big]
=\Delta t\sum_{i=1}^\infty(\Phi_1e_1^{(i)})^2,
\]
we derive 
\begin{align}
&~\E\bigg[\int_{\mathcal O}\frac{\beta}{2}|U^{n+1}|^4\diff x\bigg]
\notag
\\
=&~\E\bigg[\int_{\mathcal O}
\frac{\beta}{2}|\widetilde U^{n+1}|^4\diff x\bigg]
+\beta\Delta t\E\bigg[\int_{\mathcal O}\sum_{i=1}^{\infty}
{\rm Re}\Big((U^{n+1})^2\overline{\Phi_1e_1^{(i)}}^{\,2}\Big)
+2|U^{n+1}|^2|\Phi_1e_1^{(i)}|^2\diff x\bigg].
\label{eq:energy_aux5}
\end{align}
Finally, using \eqref{eq:detsubenergycon_new}--\eqref{eq:energy_aux5} completes the proof. 
\end{proof}

\section{Structure-preserving full discretizations based on LDG method}\label{sec:fulldiscrete}
In this section, we construct fully discrete schemes for the SSKdV equation by applying the LDG method in space. 
To introduce the LDG numerical fluxes, we first fix some notation. Let \(N_x \in \mathbb{N}\), \(h=\frac{L}{N_x}>0\), and partition the closure \(\overline{\mathcal O} = [0,L]\) into a family of uniform mesh cells \(K_j=[x_{j-\frac12},x_{j+\frac12}]\) with the cell center \(x_j=\frac{x_{j-\frac12}+x_{j+\frac12}}{2}\) and mesh size \(h_j=x_{j+\frac12}-x_{j-\frac12}\) for \(j = 1,2,\cdots,N_{x}\). Letting $\mathcal T_h := \{K_j:\ j=1,2,\cdots,N_x\}$, 
the collection of all element boundaries is denoted by \(\Gamma=\cup_{K\in\mathcal{T}_h}\partial K\), and the set of interior interfaces is defined by \(\Gamma_0 = \Gamma\setminus\partial\mathcal O\). For any function \(u\), the quantities \(u_{j+\frac12}^{-}\) and \(u_{j+\frac12}^{+}\) denote the traces of \(u\) at the interface \(x_{j+\frac12}\), taken from the left element \(K_j\) and the right element \(K_{j+1}\), respectively. 
We use the following local inner-product notation: 
\begin{align*}
\langle f,g \rangle_{K}:=\int_{K} f(x)\overline{g(x)}\,\diff x, \quad \langle f,g \rangle_{\partial K_j}:=(f\overline{g})_{j+\frac12}^{-}-(f\overline{g})_{j-\frac12}^{+},
\end{align*} 
where \(\overline{g}\) denotes the complex conjugate of \(g\). 
In addition, let \(\mathcal{P}_r^k(K)\) and \(\mathcal{P}_c^k(K)\) denote the spaces of real-valued and complex-valued polynomials of degree at most \(k\ge0\) on \(K\in\mathcal{T}_h\), respectively. Based on these local polynomial spaces, we define the finite element spaces
\[
V_h^r
:=
\left\{
\phi \in H_{\R}
:\,
\phi|_K \in \mathcal{P}_r^k(K),
\ \forall K \in \mathcal{T}_h
\right\},
\]
\[
V_h^c
:=
\left\{
\psi \in H
:\,
\psi|_K \in \mathcal{P}_c^k(K),
\ \forall K \in \mathcal{T}_h
\right\}.
\]
Functions in \(V_h^r\) and \(V_h^c\) are allowed to be discontinuous across element interfaces.

By introducing the LDG auxiliary variables
\begin{equation}\label{eq:spaceauxiliaryvs}
\begin{gathered}
      R^{n+\frac{1}{2}}
      :=
      \partial_{x}U^{n+\frac{1}{2}},
      \quad
      P^{n+\frac{1}{2}}
      :=
      \partial_{x}V^{n+\frac{1}{2}}, 
      \quad
      Q^{n+\frac{1}{2}}
      :=
      \partial_{x}P^{n+\frac{1}{2}},   
      \\
      S^{n+\frac{1}{2}}
      :=
      |U^{n+\frac{1}{2}}|^{2},
      \quad
      X^{n+\frac{1}{2}}
      :=
      \bigl(V^{n+\frac{1}{2}}\bigr)^{2},
\end{gathered}
\end{equation}
the temporal scheme \eqref{eq:temproaldis1} can be rewritten in the following equivalent first-order form
\begin{equation*}
\left\{
\begin{aligned}
      &\frac{\widetilde{U}^{n+1}-U^{n}}{\Delta t}
      =
      \i \partial_{x}R^{n+\frac{1}{2}}
      -
      \i\Big(
      \gamma_{1}U^{n+\frac{1}{2}}V^{n+\frac{1}{2}}
      +
      \beta S^{n+\frac{1}{2}}U^{n+\frac{1}{2}}
      \Big),
      \\
      &\frac{\widetilde{V}^{n+1}-V^{n}}{\Delta t}
      =
      -\partial_{x}Q^{n+\frac{1}{2}}
      +
      \gamma_{2}\partial_{x}S^{n+\frac{1}{2}}
      -
      \frac{1}{2}\partial_{x}X^{n+\frac{1}{2}},
      \\
      &R^{n+\frac{1}{2}}
      =
      \partial_{x}U^{n+\frac{1}{2}},
      \quad
      P^{n+\frac{1}{2}}
      =
      \partial_{x}V^{n+\frac{1}{2}}, 
      \quad
      Q^{n+\frac{1}{2}}
      =
      \partial_{x}P^{n+\frac{1}{2}}.
\end{aligned}
\right.
\end{equation*}
For each element $K \in \mathcal{T}_{h}$, let $\bm{\nu}$ denote the unit outward normal vector on $\partial K$. We seek the the LDG approximations
\[
U_{h}^{n}, \ \widetilde{U}_{h}^{n+1}, \ R_{h}^{n+\frac{1}{2}}, \ U_{h}^{n+\frac{1}{2}}, \ V_{h}^{n+\frac{1}{2}},
\ S_{h}^{n+\frac{1}{2}}, \ X_{h}^{n+\frac{1}{2}}, \ Q_{h}^{n+\frac{1}{2}}, \ P_{h}^{n+\frac{1}{2}}
\]
to the corresponding variables
\[
U^{n}, \ \widetilde{U}^{n+1}, \ R^{n+\frac{1}{2}}, \ U^{n+\frac{1}{2}}, \ V^{n+\frac{1}{2}},
\ S^{n+\frac{1}{2}}, \ X^{n+\frac{1}{2}}, \ Q^{n+\frac{1}{2}}, \ P^{n+\frac{1}{2}},
\]
respectively. We then define the LDG approximation of \eqref{eq:temproaldis1} as follows: find
\[
\widetilde{U}_{h}^{n+1}, \ R_{h}^{n+\frac{1}{2}} \in V_{h}^{c},
\quad
\widetilde{V}_{h}^{n+1}, \ P_{h}^{n+\frac{1}{2}}, \ Q_{h}^{n+\frac{1}{2}} \in V_{h}^{r}
\]
such that
\begin{equation*}
\left\{\begin{aligned}
      &\bigg\langle \frac{\widetilde{U}_{h}^{n+1}
      - U_{h}^{n}}{\Delta{t}}, \eta\bigg\rangle_{K}
      =
      \i\big\langle \widehat{R}_{h}^{n+\frac{1}{2}}\bm{\nu},
      \eta \big\rangle_{\partial{K}}
      - 
      \i\big\langle R_{h}^{n+\frac{1}{2}}, 
      \partial_{x}\eta \big\rangle_{K}
      \\&~~~~~~~~~~~~~~~~~~~~~~~~~~~~-
      \i\gamma_{1}\big\langle U_{h}^{n+\frac{1}{2}}
      V_{h}^{n+\frac{1}{2}}, \eta\big\rangle_{K}
      -
      \i\beta\big\langle S_{h}^{n+\frac{1}{2}}
      U_{h}^{n+\frac{1}{2}}, \eta\big\rangle_{K},
      \\
            &\bigg\langle 
      \frac{\widetilde{V}_{h}^{n+1}-V_{h}^{n}}{\Delta{t}},\phi\bigg\rangle_{K}
      =
      -\big\langle \widehat{Q}_{h}^{n+\frac{1}{2}}\bm{\nu},
      \phi \big\rangle_{\partial{K}}
      +
      \gamma_{2}\big\langle \widehat{S}_{h}^{n+\frac{1}{2}}\bm{\nu},
      \phi \big\rangle_{\partial{K}}
      -
      \frac{1}{2}\big\langle \widehat{X}_{h}^{n+\frac{1}{2}}\bm{\nu},
      \phi \big\rangle_{\partial{K}}
      \\&~~~~~~~~~~~~~~~~~~~~~~~~~~~~+
      \big\langle Q_{h}^{n+\frac{1}{2}},
      \partial_{x}\phi \big\rangle_{K}
      -
      \gamma_{2}\big\langle S_{h}^{n+\frac{1}{2}},
      \partial_{x}\phi \big\rangle_{K}
      +
      \frac{1}{2}\big\langle X_{h}^{n+\frac{1}{2}},
      \partial_{x}\phi \big\rangle_{K},
      \\
      &\big\langle R_{h}^{n+\frac{1}{2}},
      \xi\big\rangle_{K}
      =
      \big\langle \widehat{U}_{h}^{n+\frac{1}{2}}\bm{\nu},
      \xi \big\rangle_{\partial{K}}
      -
      \big\langle U_{h}^{n+\frac{1}{2}},
      \partial_{x}\xi \big\rangle_{K},
      \\
      &\big\langle P_{h}^{n+\frac{1}{2}}, 
      \varphi\rangle_{K}
      =
      \big\langle \widehat{V}_{h}^{n+\frac{1}{2}}\bm{\nu},\varphi \big\rangle_{\partial{K}}
      -
      \big\langle V_{h}^{n+\frac{1}{2}},
      \partial_{x}\varphi\big\rangle_{K}, 
      \\
      &\big\langle Q_{h}^{n+\frac{1}{2}},
      \psi \big\rangle_{K}
      =
      \big\langle \widehat{P}_{h}^{n+\frac{1}{2}}\bm{\nu},
      \psi\big\rangle_{\partial{K}}
      - \big\langle P_{h}^{n+\frac{1}{2}},
      \partial_{x} \psi\big\rangle_{K} 
\end{aligned}\right.
\end{equation*}
hold for any test function $\eta,\xi \in V_{h}^{c}$ and $\phi,\varphi,\psi \in V_{h}^{r}$. Here, the numerical fluxes 
are chosen as follows:
\begin{align*}
      \widehat{U}_{h}^{n+\frac{1}{2}} = (U_{h}^{n+\frac{1}{2}})^{-},\quad
      \widehat{R}_{h}^{n+\frac{1}{2}} = (R_{h}^{n+\frac{1}{2}})^{+},\quad 
      \widehat{V}_{h}^{n+\frac{1}{2}} = (V_{h}^{n+\frac{1}{2}})^{+},\quad
      \widehat{P}_{h}^{n+\frac{1}{2}} = (P_{h}^{n+\frac{1}{2}})^{-},
\end{align*}
and
\begin{align*}
      \widehat{Q}_{h}^{n+\frac{1}{2}} 
      = 
      \frac{(Q_{h}^{n+\frac{1}{2}})^{+} 
      + (Q_{h}^{n+\frac{1}{2}})^{-}}{2},
      \quad
      \widehat{S}_{h}^{n+\frac{1}{2}} 
      = 
      \frac{(S_{h}^{n+\frac{1}{2}})^{+} 
      + (S_{h}^{n+\frac{1}{2}})^{-}}{2},
      \quad
      \widehat{X}_{h}^{n+\frac{1}{2}} 
      = 
      \frac{(X_{h}^{n+\frac{1}{2}})^{+} 
      + (X_{h}^{n+\frac{1}{2}})^{-}}{2}.
\end{align*}
Finally, by incorporating the stochastic update, we define the numerical approximation $(U_{h}^{n+1},V_{h}^{n+1})$ by
\begin{align}
\label{FD1}
      U_{h}^{n+1}
      =
      \widetilde{U}_{h}^{n+1}
      +
      \Pi_{h}(\Phi_{1}\Delta{W_{1}^{n}}),
      \quad
      V_{h}^{n+1}
      =
      \widetilde{V}_{h}^{n+1}
      +
      \Pi_{h}(\Phi_{2}\Delta{W_{2}^{n}}),      
\end{align}
where $\Pi_h$ denotes the $L^2$-projection onto $V_h^c$ when acting on complex-valued functions and onto $V_h^r$ when acting on real-valued functions.

\begin{lemma}
Under the assumptions of Lemma \ref{lm;exact1}, for any $n\in Z_{N}$,  the full discretization \eqref{FD1} satisfies the following relationships
      \begin{align*}
            &\E\bigg[\sum_{K \in \mathcal{T}_{h}} \int_{K} |U_{h}^{n}|^{2} \diff{x}\bigg]
            =
            \E\bigg[\sum_{K \in \mathcal{T}_{h}} \int_{K} |U_{h}^{0}|^{2} \diff{x}\bigg]
            +
            \|\Phi_{1}\|_{\mathcal L_2^{0}}^{2}t_{n},\\
            &
            \E\bigg[\sum_{K \in \mathcal{T}_{h}} \int_{K} V_{h}^{n} \diff{x}\bigg]
            =
            \E\bigg[\sum_{K \in \mathcal{T}_{h}} \int_{K} V_{h}^{0} \diff{x}\bigg],\\
            &\E\bigg[\sum_{K \in \mathcal{T}_{h}} \int_{K}
            {\rm{Im}}\big(U_{h}^{n}
            \overline{\partial_{x}U_{h}^{n}}\big)
            + 
            \frac{\gamma_{1}}{2\gamma_{2}}
            \big(V_{h}^{n}\big)^{2} \diff{x}\bigg] 
            \\=&~ \notag
            \E\bigg[\sum_{K \in \mathcal{T}_{h}} \int_{K}
            {\rm{Im}}\big(U_{h}^{0}
            \overline{\partial_{x}U_{h}^{0}}\big) 
            + 
            \frac{\gamma_{1}}{2\gamma_{2}} 
            \big(V_{h}^{0}\big)^{2} \diff{x}\bigg] 
           + 
            \Big({\rm{Im}}\big({\rm{Tr}}
            (\Phi_{1}^{*}\partial_{x}\Phi_{1})\big)
            +
            \frac{\gamma_{1}}{2\gamma_{2}}
            \|\Phi_{2}\|_{\mathcal L_2^{0}}^{2}\Big)t_{n}.
      \end{align*}
\end{lemma}

The fully discrete scheme defined by the preceding LDG formulation together with the stochastic update \eqref{FD1} inherits the discrete averaged plasmon number evolution law, discrete averaged particle number conservation law, and discrete averaged momentum evolution law.  To preserve discrete averaged modified energy evolution law, we turn to the CSAV-based temporal discretization introduced in Section~\ref{sec:semidiscrete}. Using the LDG auxiliary variables introduced in \eqref{eq:spaceauxiliaryvs}, the CSAV-based temporal discretization \eqref{eq:timedis2U}--\eqref{eq:timedis2UV} can be rewritten as follows
\begin{equation*}
\left\{\begin{aligned}
      &\frac{\widetilde{U}^{n+1}-U^{n}}{\Delta{t}}
      =
      \i\partial_{x}R^{n+\frac{1}{2}}
      -
      \i r^{n+\frac{1}{2}}\big(\gamma_{1}U^{n+\frac{1}{2}}
      V^{n+\frac{1}{2}}
      +
      \beta S^{n+\frac{1}{2}}
      U^{n+\frac{1}{2}}\big),
      \\
      &\frac{\widetilde{V}^{n+1}-V^{n}}{\Delta{t}}
      =
      -\partial_{x}Q^{n+\frac{1}{2}}
      +
      \gamma_{2}r^{n+\frac{1}{2}}\partial_{x}S^{n+\frac{1}{2}}
      -
      \frac{1}{2}s^{n+\frac{1}{2}}\partial_{x}X^{n+\frac{1}{2}},
      \\
      &R^{n+\frac{1}{2}}
      =
      \partial_{x}U^{n+\frac{1}{2}},\quad P^{n+\frac{1}{2}}
      =
      \partial_{x}V^{n+\frac{1}{2}}, 
      \quad Q^{n+\frac{1}{2}}
      =
      \partial_{x}P^{n+\frac{1}{2}}. 
\end{aligned}\right.
\end{equation*}
Now we define the corresponding full discretization based on CSAV and LDG formulation: find $\widetilde{U}_{h}^{n+1}, R_{h}^{n+\frac{1}{2}} \in V_{h}^{c}$ and $\widetilde{V}_{h}^{n+1}, P_{h}^{n+\frac{1}{2}}, Q_{h}^{n+\frac{1}{2}} \in V_{h}^{r}$ such that 
\begin{equation*}
\left\{\begin{aligned}
      &\bigg\langle \frac{\widetilde{U}_{h}^{n+1}
      - U_{h}^{n}}{\Delta{t}}, \eta\bigg\rangle_{K}
      =
      \i\big\langle \widehat{R}_{h}^{n+\frac{1}{2}}\bm{\nu},
      \eta \big\rangle_{\partial{K}}
      - 
      \i\big\langle R_{h}^{n+\frac{1}{2}}, \partial_{x}\eta\big\rangle_{K}
      \\&~~~~~~~~~~~~~~~~~~~~~~~~~~~~-
      \i\gamma_{1} r^{n+\frac{1}{2}}
      \big\langle U_{h}^{n+\frac{1}{2}}
      V_{h}^{n+\frac{1}{2}}, \eta\big\rangle_{K}
      -
      \i\beta r^{n+\frac{1}{2}}
      \big\langle S_{h}^{n+\frac{1}{2}}
      U_{h}^{n+\frac{1}{2}}, \eta\big\rangle_{K},
      \\
       &\bigg\langle \frac{\widetilde{V}_{h}^{n+1}
      - V_{h}^{n}}{\Delta{t}}, \phi\bigg\rangle_{K}
      =
      -\big\langle \widehat{Q}_{h}^{n+\frac{1}{2}}\bm{\nu},
      \phi \big\rangle_{\partial{K}}
      +
      \gamma_{2} r^{n+\frac{1}{2}}
      \big\langle \widehat{S}_{h}^{n+\frac{1}{2}}\bm{\nu},
      \phi \big\rangle_{\partial{K}}
      -
      \frac{1}{2}s^{n+\frac{1}{2}}
      \big\langle \widehat{X}_{h}^{n+\frac{1}{2}}\bm{\nu},
      \phi \big\rangle_{\partial{K}}
      \\&~~~~~~~~~~~~~~~~~~~~~~~~~~~~+
      \big\langle Q_{h}^{n+\frac{1}{2}},
      \partial_{x}\phi \big\rangle_{K}
      -
      \gamma_{2} r^{n+\frac{1}{2}}
      \big\langle S_{h}^{n+\frac{1}{2}},
      \partial_{x}\phi \big\rangle_{K}
      +
      \frac{1}{2}s^{n+\frac{1}{2}}
      \big\langle X_{h}^{n+\frac{1}{2}},
      \partial_{x}\phi \big\rangle_{K},
      \\
      &\big\langle R_{h}^{n+\frac{1}{2}},
      \xi \big\rangle_{K}
      =
      \big\langle \widehat{U}_{h}^{n+\frac{1}{2}}\bm{\nu},
      \xi \big\rangle_{\partial{K}}
      -
      \big\langle U_{h}^{n+\frac{1}{2}},
      \partial_{x}\xi \big\rangle_{K},
      \\
      &\big\langle P_{h}^{n+\frac{1}{2}},
      \varphi\rangle_{K}
      =
      \big\langle \widehat{V}_{h}^{n+\frac{1}{2}}\bm{\nu},\varphi \big\rangle_{\partial{K}}
      -
      \big\langle V_{h}^{n+\frac{1}{2}},
      \partial_{x}\varphi \big\rangle_{K}, 
      \\
      &\big\langle Q_{h}^{n+\frac{1}{2}},
      \psi \big\rangle_{K}
      =
      \big\langle \widehat{P}_{h}^{n+\frac{1}{2}}\bm{\nu},
      \psi\big\rangle_{\partial{K}}
      - \big\langle P_{h}^{n+\frac{1}{2}},
      \partial_{x} \psi \big\rangle_{K} 
\end{aligned}\right.
\end{equation*}
hold for any test function $\eta,\xi \in V_{h}^{c}$ and $\phi,\varphi,\psi \in V_{h}^{r}$, where the numerical fluxes $$ \widehat{U}_{h}^{n+\frac{1}{2}}, \widehat{R}_{h}^{n+\frac{1}{2}}, \widehat{V}_{h}^{n+\frac{1}{2}}, \widehat{P}_{h}^{n+\frac{1}{2}}, \widehat{Q}_{h}^{n+\frac{1}{2}}, \widehat{S}_{h}^{n+\frac{1}{2}}, \widehat{X}_{h}^{n+\frac{1}{2}}$$ 
are given by
\begin{align*}
      \widehat{U}_{h}^{n+\frac{1}{2}} = (U_{h}^{n+\frac{1}{2}})^{-},\quad
      \widehat{R}_{h}^{n+\frac{1}{2}} = (R_{h}^{n+\frac{1}{2}})^{+},\quad 
      \widehat{V}_{h}^{n+\frac{1}{2}} = (V_{h}^{n+\frac{1}{2}})^{+},\quad
      \widehat{P}_{h}^{n+\frac{1}{2}} = (P_{h}^{n+\frac{1}{2}})^{-},
\end{align*}
and
\begin{align*}
      \widehat{Q}_{h}^{n+\frac{1}{2}} 
      = 
      \frac{(Q_{h}^{n+\frac{1}{2}})^{+} 
      + (Q_{h}^{n+\frac{1}{2}})^{-}}{2},
      \quad
      \widehat{S}_{h}^{n+\frac{1}{2}} 
      = 
      \frac{(S_{h}^{n+\frac{1}{2}})^{+} 
      + (S_{h}^{n+\frac{1}{2}})^{-}}{2},
      \quad
      \widehat{X}_{h}^{n+\frac{1}{2}} 
      = 
      \frac{(X_{h}^{n+\frac{1}{2}})^{+} 
      + (X_{h}^{n+\frac{1}{2}})^{-}}{2}.
\end{align*}
The auxiliary variables $r^{n+1}$ and $s^{n+1}$ are updated by
\begin{align*}
      \frac{r^{n+1}-r^{n}}{\Delta{t}} 
      =&~ 
      \alpha_{1} \bigg( -\frac{1}{\Delta{t}}
      \sum_{K \in \mathcal{T}_{h}} \int_{K} 
      \bigg(\frac{\beta}{2}|\widetilde{U}_{h}^{n+1}|^4 
      + \gamma_1 |\widetilde{U}_{h}^{n+1}|^{2} \widetilde{V}_{h}^{n+1}\bigg)
      \\&~- \bigg(\frac{\beta}{2}|U_{h}^{n}|^4 
      + \gamma_1 |U_{h}^{n}|^{2} V_{h}^{n}\bigg) \diff{x} 
      + r^{n+\frac{1}{2}} P_{1,h}^{n+\frac{1}{2}} \bigg),
\end{align*}
and
\begin{equation*}
      \frac{s^{n+1}-s^{n}}{\Delta{t}}  
      = 
      \alpha_2 \bigg(\frac{\gamma_1}{6\gamma_2} \frac{1}{\Delta{t}}
      \int_{\mathcal O} (\widetilde{V}_{h}^{n+1})^{3} - (V_{h}^{n})^{3} \diff{x} 
      + s^{n+\frac{1}{2}} P_{2,h}^{n+\frac{1}{2}} \bigg),
\end{equation*}
where
\begin{align*}
      P_{1,h}^{n+\frac{1}{2}} 
      &= 
      2\beta \Re\sum_{K \in \mathcal{T}_{h}} \int_{K}
      |U_{h}^{n+\frac{1}{2}}|^{2} U_{h}^{n+\frac{1}{2}} 
      \overline{\frac{\widetilde{U}_{h}^{n+1}
      - U_{h}^{n}}{\Delta{t}}} \diff{x} 
      \\&\quad 
      + 2\gamma_{1} \Re\sum_{K \in \mathcal{T}_{h}} \int_{K}
      U_{h}^{n+\frac{1}{2}} V_{h}^{n+\frac{1}{2}} 
      \overline{\frac{\widetilde{U}_{h}^{n+1}
      - U_{h}^{n}}{\Delta{t}}} \diff{x} 
      +\gamma_{1} \sum_{K \in \mathcal{T}_{h}} 
      \int_{K} |U_{h}^{n+\frac{1}{2}}|^2 
      \frac{\widetilde{V}_{h}^{n+1}
      -V_{h}^{n}}{\Delta{t}} \diff{x},
\end{align*}
and
\begin{align*}
      P_{2,h}^{n+\frac{1}{2}} 
      &= 
      -\frac{\gamma_1}{2\gamma_2} 
      \sum_{K \in \mathcal{T}_{h}} 
      \int_{K} (V_{h}^{n+\frac{1}{2}})^{2} 
      \frac{\widetilde{V}_{h}^{n+1}
      -V_{h}^{n}}{\Delta{t}} \diff{x}.
\end{align*}
By incorporating the stochastic update through the projection operator $\Pi_{h}$, we define the full discretization via the approximation $(U_{h}^{n+1},V_{h}^{n+1})$ given by
\begin{align}
\label{fd2}
      U_{h}^{n+1}
      =
      \widetilde{U}_{h}^{n+1}
      +
      \Pi_{h}(\Phi_{1}\Delta{W_{1}^{n}}),
      \quad
      V_{h}^{n+1}
      =
      \widetilde{V}_{h}^{n+1}
      +
      \Pi_{h}(\Phi_{2}\Delta{W_{2}^{n}}).
\end{align}
 \begin{lemma}
       Under the assumptions of Lemma \ref{thm:existunique}, for any $n\in Z_{N}$, the full discretization \eqref{fd2} satisfies
      \begin{align*}
            &~\E\bigg[\sum_{K \in \mathcal{T}_{h}} 
            \int_{K} |\partial_{x}U_{h}^{n}|^{2}
            + \frac{\beta}{2}|U_{h}^{n}|^{4} 
            + \gamma_{1}|U_{h}^{n}|^{2}V_{h}^{n}
            -
            \frac{\gamma_{1}}{6\gamma_{2}}(V_{h}^{n})^{3} 
            + 
            \frac{\gamma_{1}}{2\gamma_{2}}|
            \partial_{x}V_{h}^{n}|^{2} \diff{x}\bigg]
        \\=&~
            \E\bigg[\sum_{K \in \mathcal{T}_{h}} \int_{K}
            |\partial_{x}U_{h}^{0}|^{2} 
            +
            \frac{\beta}{2}|U_{h}^{0}|^{4}
            +
            \gamma_{1}|U_{h}^{0}|^{2}V_{h}^{0}
            -
            \frac{\gamma_{1}}{6\gamma_{2}}(V_{h}^{0})^{3}
            +
            \frac{\gamma_{1}}{2\gamma_{2}}
             |\partial_{x}V_{h}^{0}|^{2} \diff{x}\bigg]
            \\&~+
            \bigg(\|\Phi_{1}\|_{\mathcal L_2^{1}}^{2}
            +
            \frac{\gamma_{1}}{2\gamma_{2}}
            \|\Phi_{2}\|_{\mathcal L_2^{1}}^{2}\bigg)t_{n}
            +
            \gamma_{1}\sum_{j=0}^{n-1}
            \E\bigg[\sum_{K \in \mathcal{T}_{h}} \int_{K} 
            V_{h}^{j+1} \sum_{i=1}^{\infty}|\Phi_{1}e_{1}^{(i)}|^{2}
            \diff{x}\bigg]\Delta{t}
            \\&~
            -
            \frac{\gamma_{1}}{2\gamma_{2}}
            \sum_{j=0}^{n-1}\E\bigg[\sum_{K \in \mathcal{T}_{h}} \int_{K}
             V_{h}^{j+1} \sum_{i=1}^{\infty}(\Phi_{2}e_{2}^{(i)})^{2}
            \diff{x}\bigg]\Delta{t}
            +
            \frac{r^{0} - r^{n}}{\alpha_{1}} 
            + \frac{s^{0} - s^{n}}{\alpha_{2}}
            \\&~+
            \beta\Delta{t}\sum_{j=0}^{n-1}
            \E\bigg[\sum_{K \in \mathcal{T}_{h}} \int_{K}
            \sum_{i=1}^{\infty} {\rm{Re}}\bigg((U_{h}^{j+1})^{2}
            \overline{\Phi_{1}e_{1}^{(i)}}^{2}\bigg) 
            +
            2|U_{h}^{j+1}|^{2} |\Phi_{1}e_{1}^{(i)}|^{2}\diff{x}\bigg].
      \end{align*}   
\end{lemma}

\section{Numerical experiments}\label{sec:nmuexperiments}
In this section, we present a series of numerical experiments to illustrate the theoretical results. In particular, we verify the evolution laws of the main physical quantities and investigate the strong convergence behavior of the proposed numerical discretizations for the SSKdV equation \eqref{eq:SKeq}. For comparison, we also present numerical results for the deterministic Schr\"{o}dinger--KdV equation \eqref{eq:detsubsystem}. Throughout the experiments, we take $\beta = 1$, $\gamma_{1} = \gamma_{2} = 1$, $\mathcal{O} = (0,L)$ with $L = 2\pi$, and the initial values $u^{0}(x) = e^{-(x-\pi)^{2}}e^{\i x}, v^{0}(x) = \sin x, x \in [0,2\pi]$. 

For the complex-valued Schr\"{o}dinger component, we use the complex Fourier modes 
$$e_{i}(x) = \frac{1}{\sqrt{L}} e^{\i \frac{2\pi i}{L} x}, 
\quad x \in \mathcal{O}, i \in \mathbb{Z} \backslash \{0\}$$
as the range modes of the noise operator, where $\mathbb{Z}$ denotes the set of integers. Let $\{g_{i}\}_{i \in \mathbb{Z} \backslash \{0\}}$ be an orthonormal basis of $H_{\R}$. By defining $\Phi_{1} g_{i} = \sigma_{i}^{(1)}e_{i}, i \in \mathbb{Z} \backslash \{0\}$, $\Phi_{1}$ maps real-valued noise directions into the complex-valued space $H^{1}$ with coefficients $\{\sigma_{i}^{(1)}\}_{i \neq 0} \subset \R$ satisfying
$$\|\Phi_{1}\|_{\mathcal L_2^{1}}^{2} = 
  \sum_{i \neq 0} \bigg(1 + \Big(\frac{2\pi i}{L}\Big)^{2}\bigg)  
  |\sigma_{i}^{(1)}|^{2} < \infty.$$
Besides, for the real-valued KdV component, the noise is expanded in a real Fourier basis consisting of sine and cosine functions. This choice preserves both real-valuedness and spatial translation invariance of the stochastic forcing. Let $e_{0}(x) = \frac{1}{\sqrt{L}}$ and 
\begin{align*}
      e_{i}^{(c)}(x) = \sqrt{\frac{2}{L}}\cos\frac{2\pi i x}{L},
      \quad
      e_{i}^{(s)}(x) = \sqrt{\frac{2}{L}}\sin\frac{2\pi i x}{L},
      \quad x \in \mathcal{O}, i \in \N      
\end{align*}
form the real Fourier basis of $H_{\R}$. In the present simulations, the constant mode is not forced, that is, the coefficient of $e_{0}$ is set to zero. The corresponding noise operator $\Phi_{2} \colon H_{\R} \to H_{\R}^{1}$ is defined through its real Fourier representation
\begin{align*}
      \Phi_{2}\varphi = \sum_{i=1}^{\infty} \sigma_{i}^{(2)}
      \big(\langle \varphi, e_{i}^{(c)} \rangle e_{i}^{(c)}
      + \langle \varphi, e_{i}^{(s)} \rangle e_{i}^{(s)}\big)
\end{align*}
with coefficients $\{\sigma_{i}^{(2)}\}_{i \in \N} \subset \R$ satisfying
\begin{align*}
      \sum_{i=1}^{\infty}2\bigg(1 + \Big(\frac{2\pi i}{L}\Big)^{2}\bigg)|\sigma_{i}^{(2)}|^{2} < \infty.
\end{align*}
Moreover, let $\{W_{1}(t)\}_{t \geq 0}$ and $\{W_{2}(t)\}_{t \geq 0}$ be two independent cylindrical Wiener processes on $H_{\R}$, represented via the above Fourier bases as follows
\begin{align*}
      W_{1}(t) = \sum_{i \neq 0} \beta_{i}^{(1)}(t)g_{i},
      \quad
      W_{2}(t) = \sum_{i=1}^{\infty} \big(\beta_{i}^{(2,c)}(t)e_{i}^{(c)}
      + \beta_{i}^{(2,s)}(t)e_{i}^{(s)}\big), \quad t \geq 0.  
\end{align*}
Here, all the sequences of scalar Brownian motions $\{\beta_{i}^{(1)}\}_{i \neq 0}$,
$\{\beta_{i}^{(2,c)}\}_{i \in \N}$ and $\{\beta_{i}^{(2,s)}\}_{i \in \N}$ are assumed to be mutually independent.

\begin{figure}[!htbp]
\begin{center}
      \subfigure[Evolution law of plasmon number]
      {\includegraphics[width = 6.75cm, height = 3.75cm]
      {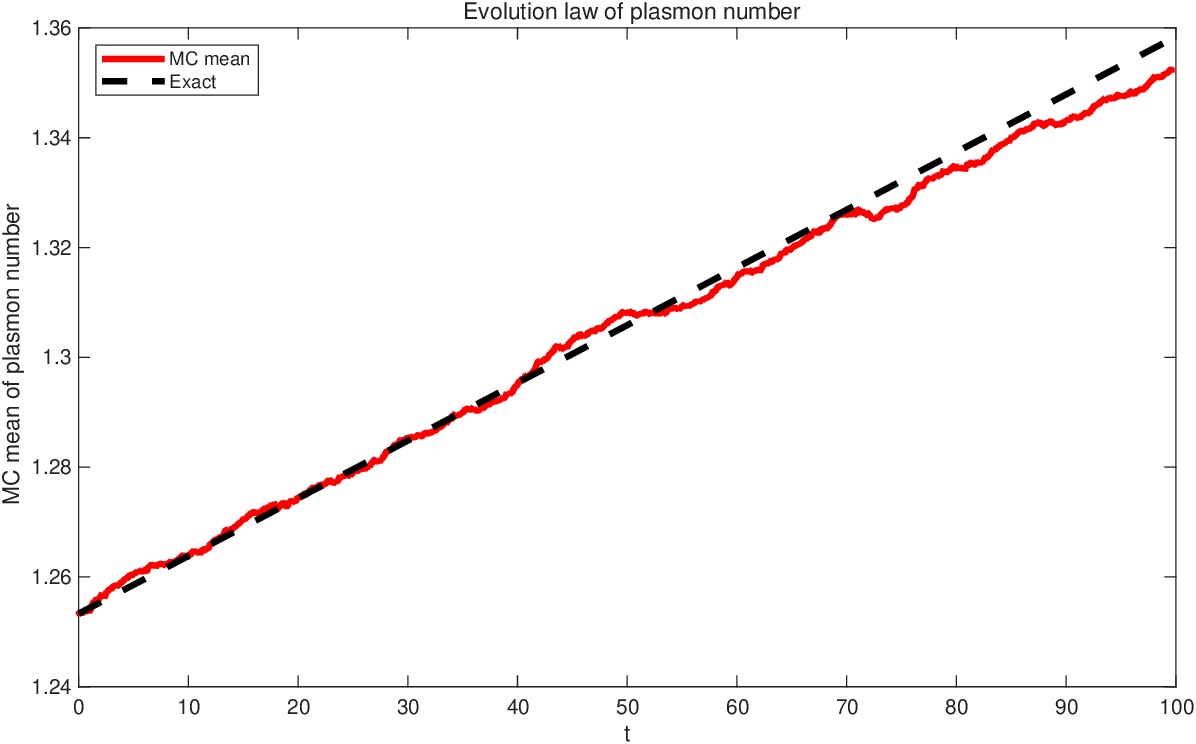}}
      \quad
      \subfigure[Evolution law of particle number]
      {\includegraphics[width = 6.75cm, height = 3.75cm]
      {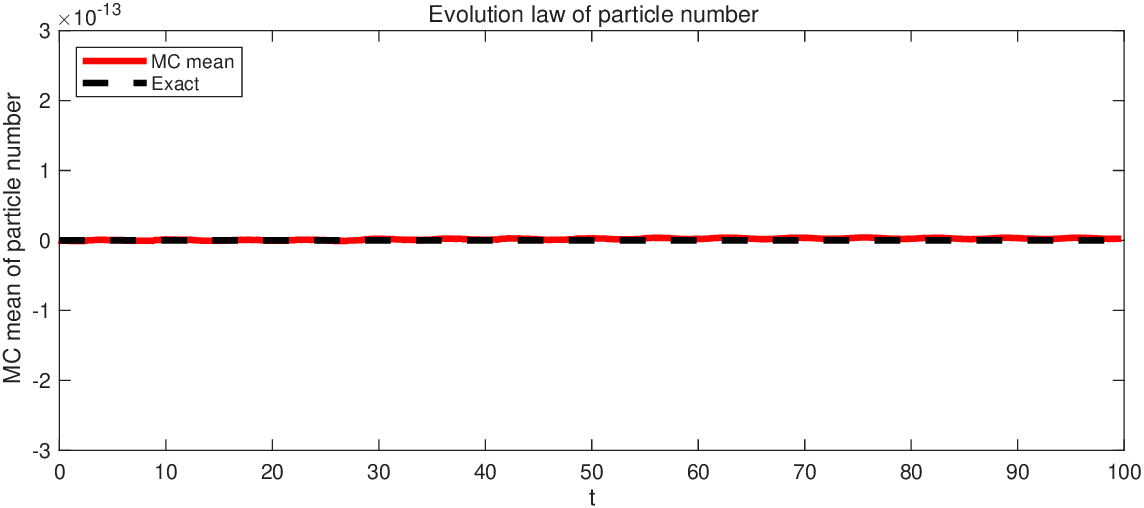}}
      \\
      \subfigure[Evolution law of momentum]
      {\includegraphics[width = 6.75cm, height = 3.75cm]
      {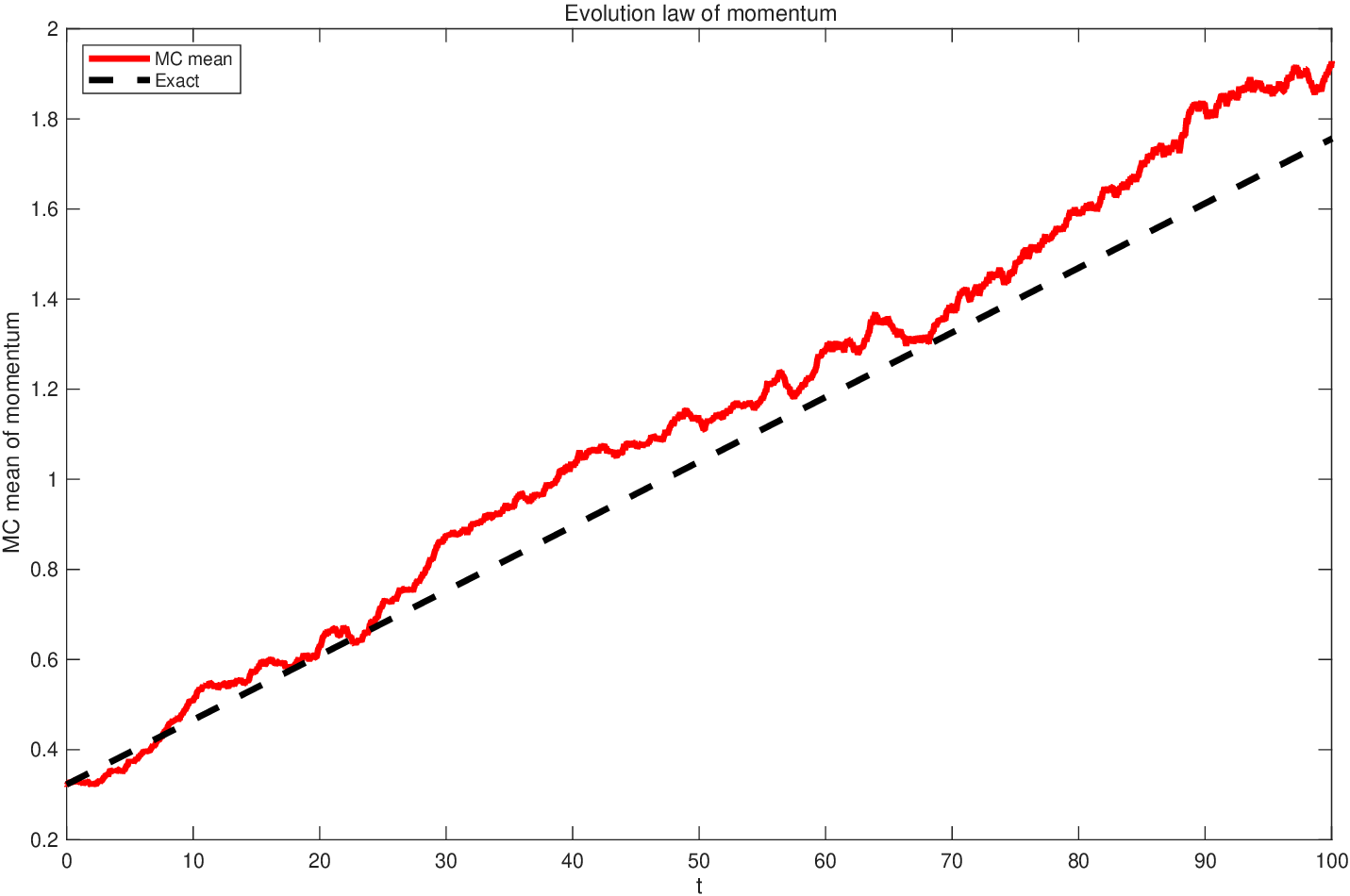}}
      \quad
      \subfigure[Evolution law of modified energy]
      {\includegraphics[width = 6.75cm, height = 3.75cm]
      {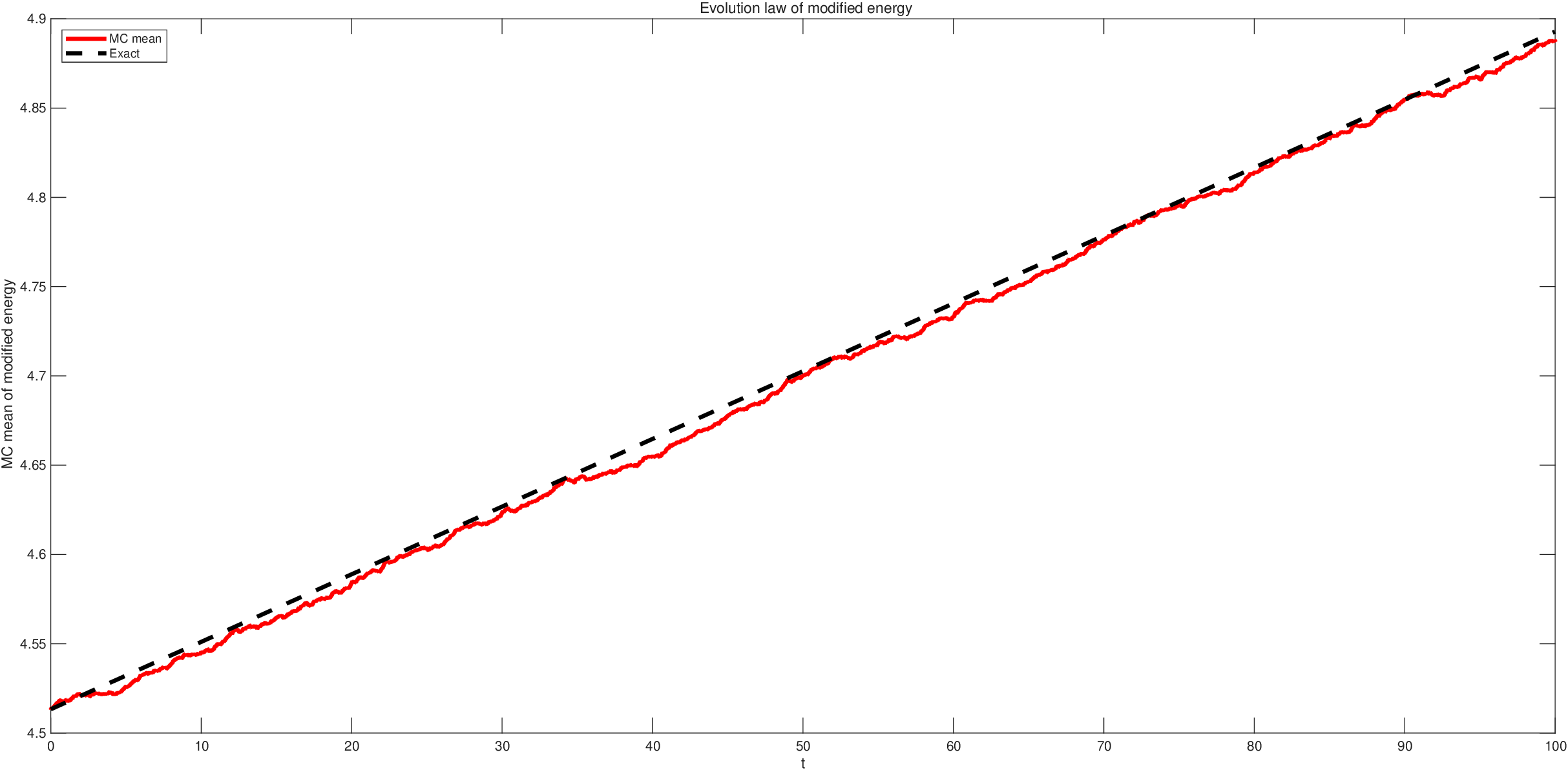}}
      \caption{Evolution laws for SSKdV equation  \eqref{eq:SKeq}}
      \label{fig:evolutionlaws}
\end{center}
\end{figure}

In the numerical simulations, the noise operators $\Phi_{1}$ and $\Phi_{2}$ are approximated by finite-dimensional truncations of their Fourier representations with  $\sigma_{i}^{(1)} = \sigma_{1}(1+|i|)^{-p_{1}}$ for $-K_{1} \leq i \leq K_{2}, i \neq 0$, and $\sigma_{i}^{(2)} = \sigma_{2}(1+|i|)^{-p_{2}}$ for $1 \leq i \leq K_{3}$. Here, the parameters $\sigma_{1} > 0, \sigma_{2} > 0, p_{1} > \frac{3}{2}, p_{2} > \frac{3}{2}$ and the truncation levels $K_{1}, K_{2}, K_{3} \geq 1$ are selected so that $\Phi_{1} \colon H_{\R} \to H^{1}$ and $\Phi_{2} \colon H_{\R} \to H_{\R}^{1}$ are Hilbert--Schmidt operators. In the experiments below, we use the parameter values 
\begin{align*}
      \sigma_{1} = 0.08, \quad
      \sigma_{2} = 0.06, \quad
      p_{1} = p_{2} = 2, \quad
      K_{1} = K_{2} = K_{3} = 8.
\end{align*}
Since $\partial_{x}e_{i} = \i \frac{2\pi i}{L} e_{i}$, and $\Phi_{1} g_{i} = \sigma_{i}^{(1)}e_{i}$ for $i \neq 0$, one obtains 
$${\rm{Tr}} (\Phi_{1}^{*}\partial_{x}\Phi_{1}) 
= \sum_{i \neq 0} \big\langle \Phi_{1}^{*}\partial_{x}\Phi_{1} g_{i}, g_{i} \big\rangle 
= \sum_{i \neq 0} \left\langle \i\frac{2\pi i}{L}\sigma_{i}^{(1)}e_{i},
  \sigma_{i}^{(1)}e_{i} \right\rangle  
= \i \sum_{i \neq 0} \frac{2\pi i}{L} |\sigma_{i}^{(1)}|^{2},$$
and hence  
\begin{align*}
      {\rm{Im}}\big({\rm{Tr}}
      (\Phi_{1}^{*}\partial_{x}\Phi_{1})\big)
      =  
      \sum_{i \neq 0} \frac{2\pi i}{L} 
      |\sigma_{i}^{(1)}|^{2}.
\end{align*}
Under the symmetric truncation $K_{1} = K_{2}$ and $\sigma_{-i}^{(1)} = \sigma_{i}^{(1)}$, the positive and negative contributions cancel pairwise, and hence ${\rm{Im}}\big({\rm{Tr}}(\Phi_{1}^{*}\partial_{x}\Phi_{1})\big) = 0$. 
In addition, the increments of cylindrical Wiener process are projected onto the finite element spaces via the projection operator $\Pi_{h}$. At each time step $t_{n}$, the stochastic updates take the form
\begin{gather*}
      \Pi_{h}(\Phi_{1} \Delta{W_{1}^{n}})
      =
      \sum_{\substack{-K_1\leq i \leq K_2\\ i \neq 0}}
      \sigma_{i}^{(1)} \Pi_{h}(e_{i}) \Delta\beta_{i}^{(1,n)},
      \\
      \Pi_{h}(\Phi_{2} \Delta{W_{2}^{n}})
      =
      \sum_{1 \leq i \leq K_{3}} \sigma_{i}^{(2)} 
      \big(\Pi_{h}(e_{i}^{(c)}) \Delta\beta_{i}^{(2,c,n)}
      + \Pi_{h}(e_{i}^{(s)}) \Delta\beta_{i}^{(2,s,n)} \big)     
\end{gather*}
with $\Delta\beta_{i}^{(1,n)} := \beta_{i}^{(1)}(t_{n+1}) - \beta_{i}^{(1)}(t_{n})$,
$\Delta\beta_{i}^{(2,c,n)} := \beta_{i}^{(2,c)}(t_{n+1}) - \beta_{i}^{(2,c)}(t_{n})$,
and $\Delta\beta_{i}^{(2,s,n)} := \beta_{i}^{(2,s)}(t_{n+1}) - \beta_{i}^{(2,s)}(t_{n})$.
These increments are mutually independent normal random variables with mean zero and variance $\Delta{t}$.

\begin{figure}[h]
\begin{center}
      \subfigure[Conservation law of plasmon number]
      {\includegraphics[width = 6.75cm, height = 3.75cm]
      {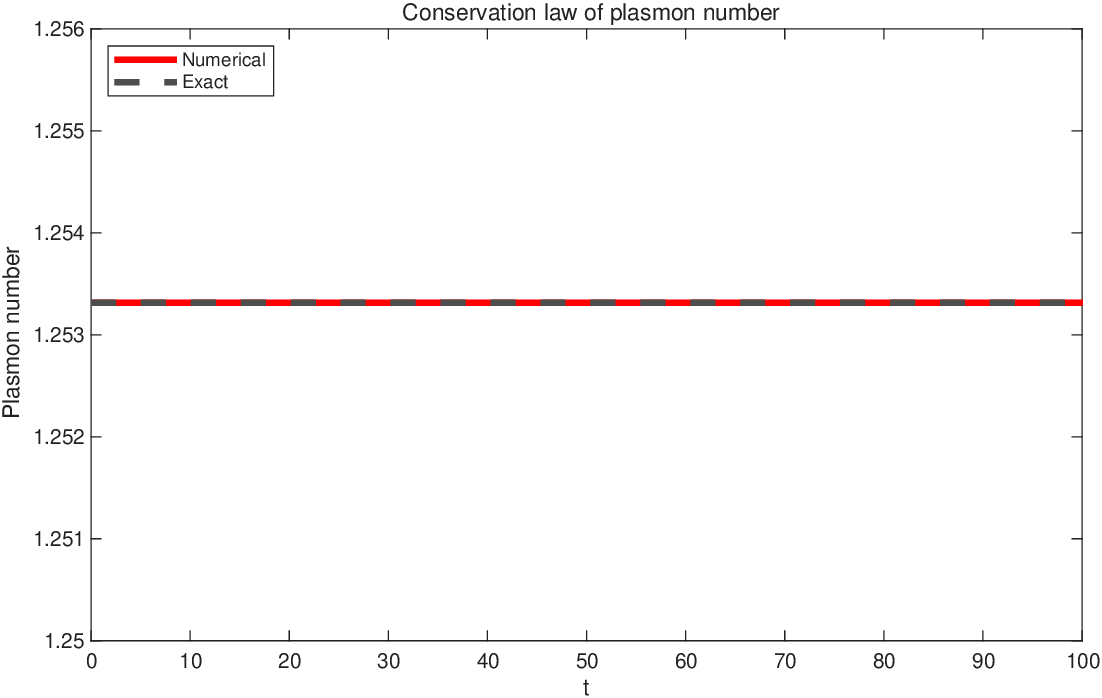}}
      \quad
      \subfigure[Conservation law of particle number]
      {\includegraphics[width = 6.75cm, height = 3.75cm]
      {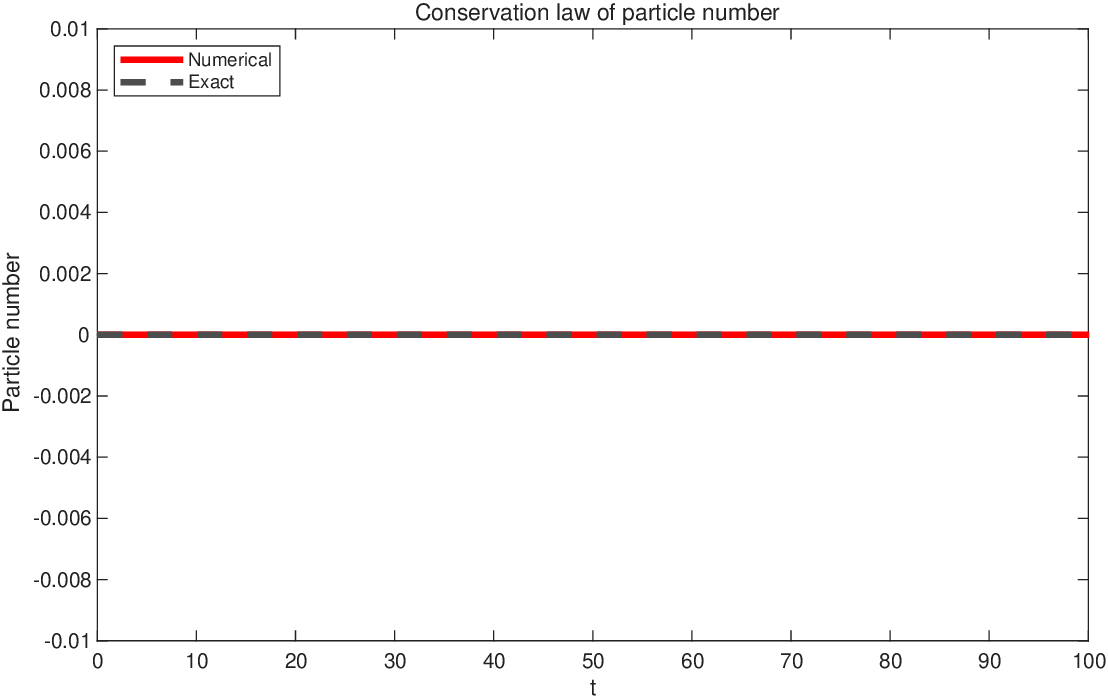}}
      \\
      \subfigure[Conservation law of momentum]
      {\includegraphics[width = 6.75cm, height = 3.75cm]
      {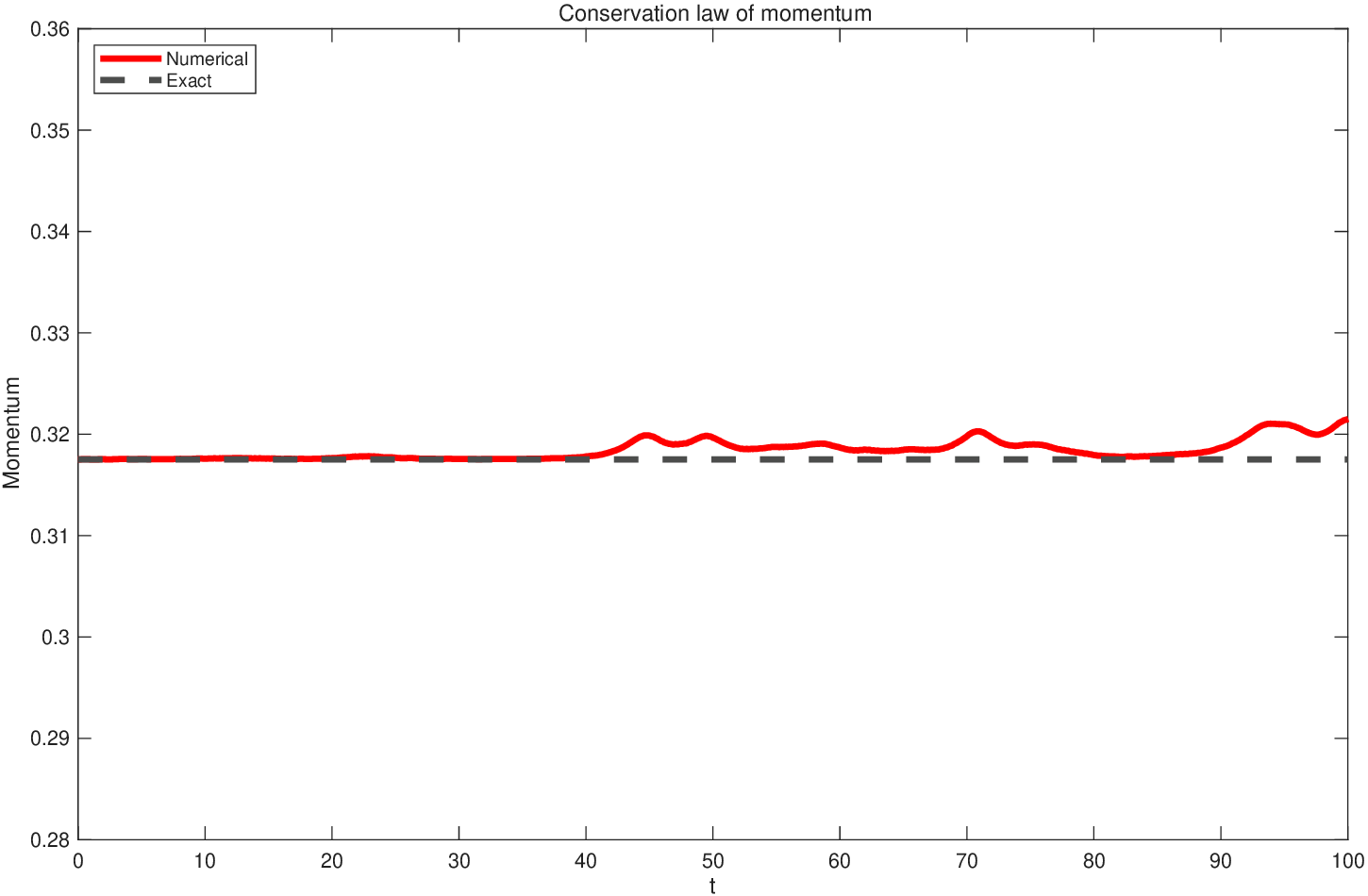}}
      \quad
      \subfigure[Conservation law of modified energy]
      {\includegraphics[width = 6.75cm, height = 3.75cm]
      {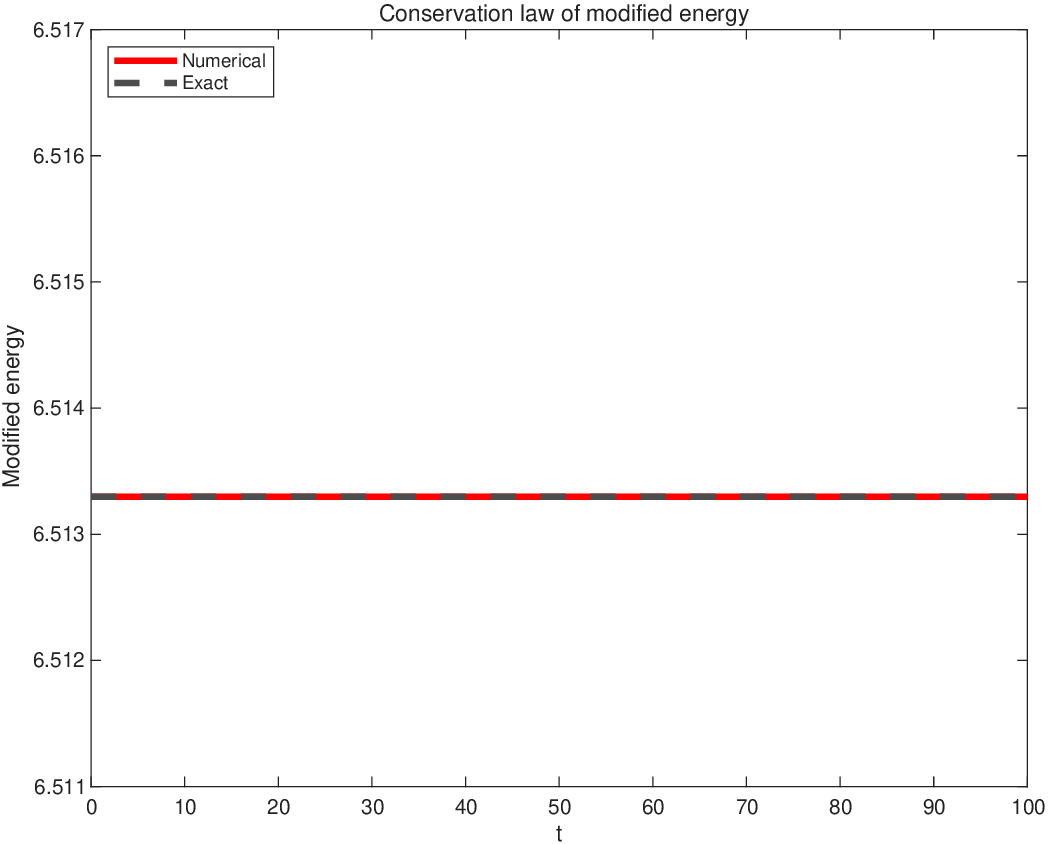}}
      \caption{Conservation laws for Schr\"{o}dinger--KdV equation \eqref{eq:detsubsystem}}
      \label{fig:conservationlaws}
\end{center}
\end{figure}

\begin{figure}[h]
\begin{center}
      \subfigure[Convergence order of \eqref{eq:temproaldis1}]
      {\includegraphics[width = 6.75cm, height = 3.75cm]
      {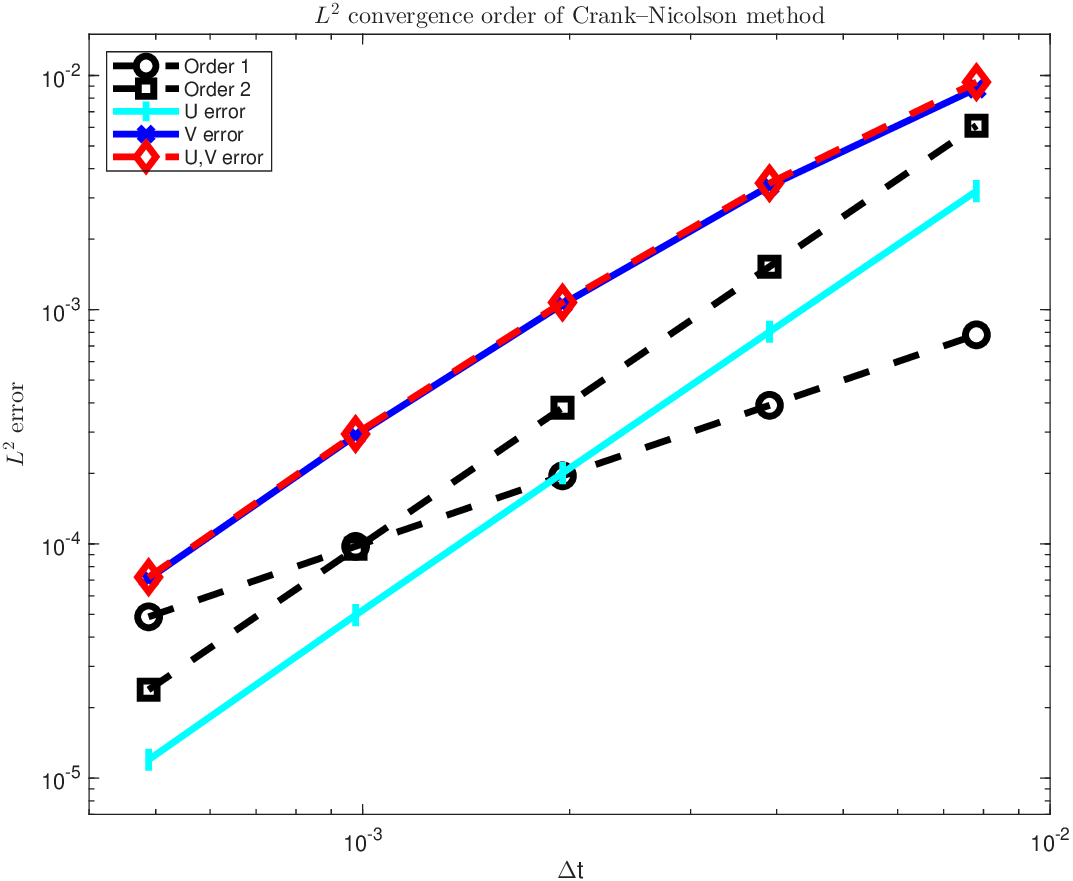}}
      \quad
      \subfigure[Convergence order of \eqref{eq:timedis2U}]
      {\includegraphics[width = 6.75cm, height = 3.75cm]
      {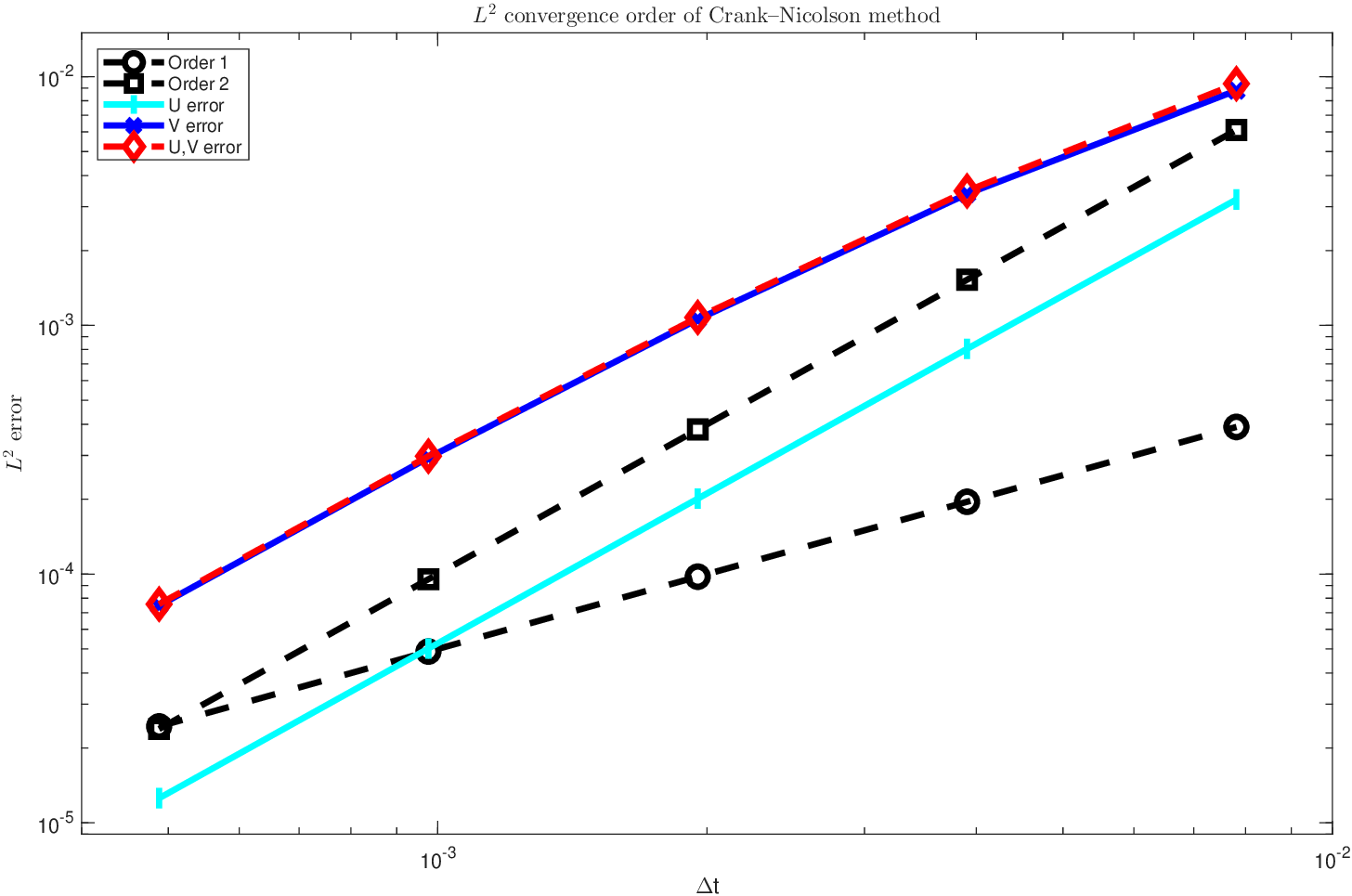}}
      \\
      \subfigure[Strong convergence order of \eqref{eq:temproaldis1UVn1}]
      {\includegraphics[width = 6.75cm, height = 3.75cm]
      {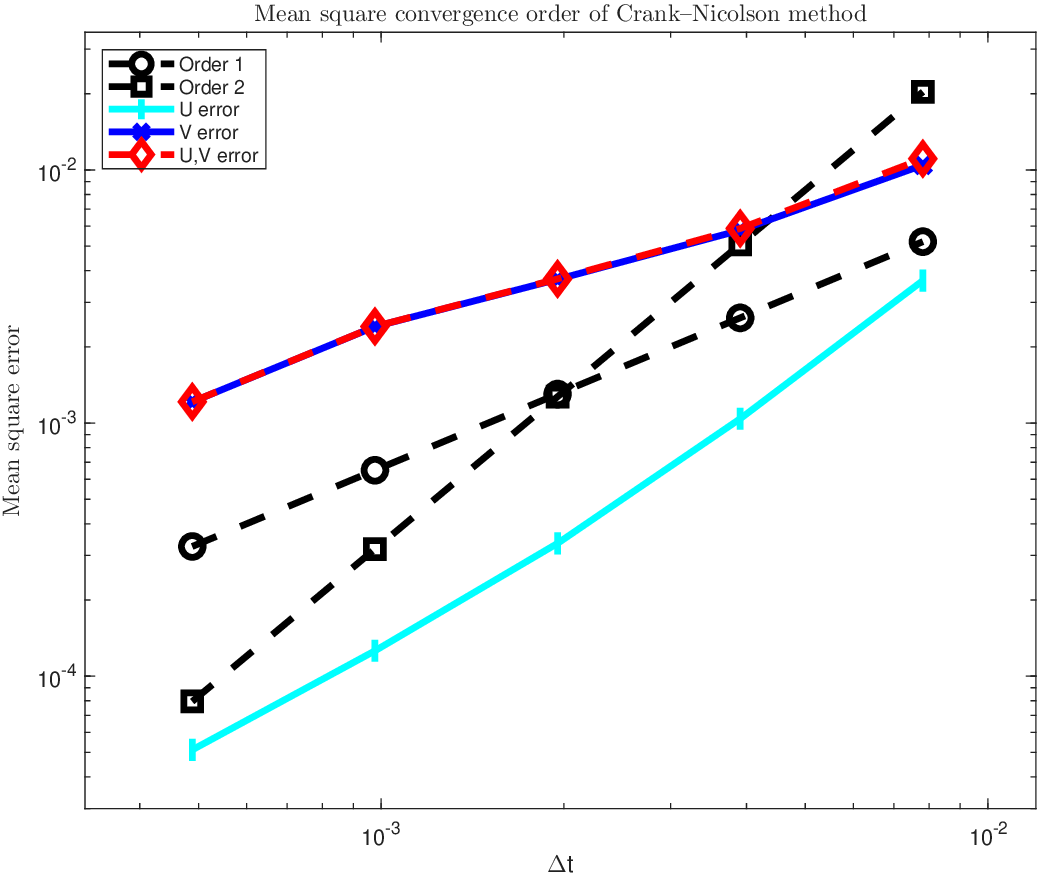}}
      \quad
      \subfigure[Strong convergence order of  \eqref{eq:timedis2UV}]
      {\includegraphics[width = 6.75cm, height = 3.75cm]
      {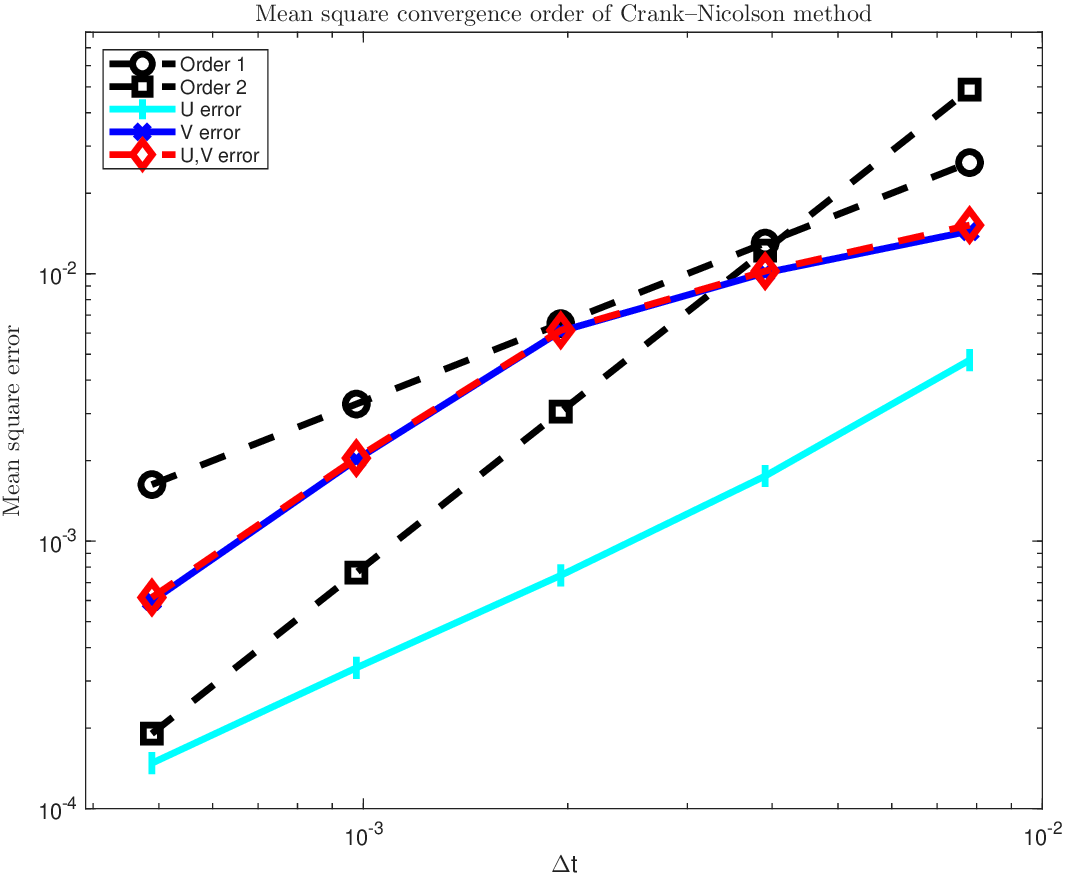}}
      \caption{Convergence orders of numerical discretizaitons for \eqref{eq:detsubsystem} and \eqref{eq:SKeq}}
      \label{fig:convergenceorders}
\end{center}
\end{figure}



In all numerical experiments, the coupled nonlinear algebraic systems generated by the proposed discretizations are solved at each time step by a Picard iteration, with a stopping tolerance of  $10^{-12}$ and a maximum of $80$ iterations. For the spatial approximation, we employ the LDG method with piecewise constant (P0) basis functions on a uniform partition of $[0,L]$ into $N_{x} = 80$ cells, subject to periodic boundary conditions; see, e.g., \cite{larson2013finite, riviere2008discontinuous}. The expectation is approximated by Monte Carlo simulation using $1000$  independent  sample paths.

Figure \ref{fig:evolutionlaws} displays the evolution laws of four key physical quantities for the stochastic equation \eqref{eq:SKeq} on the time interval $[0,100]$, obtained by the fully discrete schemes. The time stepsize is chosen as $\Delta{t} = 10^{-2}$ for panels (a), (b), and (d), and $\Delta{t} = 10^{-3}$ for panel (c). As shown in Figure \ref{fig:evolutionlaws}, the MC mean of the particle number remains nearly constant over time in panel (b), in excellent agreement with the theoretical conservation law. Meanwhile, the MC means of the plasmon number, momentum, and modified energy in panels (a), (c), and (d), respectively, exhibit the expected linear growth behavior and closely follow the corresponding theoretical reference lines. Under the same numerical setting, Figure \ref{fig:conservationlaws} presents the conservation laws of the corresponding four physical quantities for the deterministic equation \eqref{eq:detsubsystem}. The numerical results show that the proposed fully discrete schemes preserve these quantities up to very small discretization errors, confirming the expected structure-preserving property of the methods in the deterministic case.

We next investigate the convergence behavior of the temporal discretizaitons. Since the exact solution $(u(T),v(T))$ at $T = 1$ is unavailable, we use a highly resolved numerical solution $(U_{h}^{N},V_{h}^{N})$, computed with the sufficiently small time stepsize $\Delta{t} = 2^{-13} = T/N, N \in \N$, as the reference solution. The remaining numerical approximations are computed with five larger time stepsizes $\Delta{t} = 2^{-i}, i = 7,8,9,10,11$. Panels (a) and (b) in Figure \ref{fig:convergenceorders} show the error convergence, measured in the $L^2$-norm, of  \eqref{eq:temproaldis1} and  \eqref{eq:timedis2U} for the deterministic equation \eqref{eq:detsubsystem}, respectively. In both cases, the error curves are parallel to the reference line of slope of $2$, indicating second-order convergence in time. For the stochastic equation \eqref{eq:SKeq}, panels (c) and (d) of Figure \ref{fig:convergenceorders} illustrate the mean-square convergence behavior of \eqref{eq:temproaldis1UVn1} and \eqref{eq:timedis2UV}, respectively. The numerical results suggest that the proposed temporal discretizations are strongly convergent with order $1$.

\section{Conclusions and further work}\label{sec:conclusion}
In this paper, we investigated the SSKdV equation driven by additive noise. We first derived evolution laws for several fundamental averaged physical quantities, including the plasmon number, particle number, momentum, and a modified energy functional. Based on these continuous properties, we constructed two classes of structure-preserving temporal discretizations. The first temporal discretization preserves the discrete averaged evolution laws for the plasmon number and momentum, together with the conservation law of the averaged particle number, while the second one, based on the CSAV reformulation, yields a discrete averaged evolution law for the modified energy functional. By combining these temporal discretizations with the LDG method in space, we further derived the corresponding fully discrete schemes and proved that they inherit the desired discrete physical evolution laws. Numerical experiments confirmed the theoretical predictions and demonstrated the accuracy, robustness, and structure-preserving properties of the proposed methods.

Several problems deserve further investigation. In particular, it would be interesting to establish a rigorous convergence analysis for the proposed fully discrete schemes, especially the mean-square convergence order in the stochastic case. It is also worthwhile to extend the present framework to more general SSKdV type models, such as systems with multiplicative noise or more general nonlinear interactions.

\section*{Acknowledgement}

This work is supported by MOST National Key R$\&$D Program No. 2024FA1015900, the National Natural Science Foundation of China (No. 12471386, No. 12461160278, No. 12201552) and Yunnan Fundamental Research Projects (No. 202601AT070161).








\bibliographystyle{plain}	
\bibliography{BibTeXSSKdV}
\end{document}